\documentclass[12pt]{amsart}
\numberwithin{equation}{section}

\usepackage{amsmath}
\usepackage{amscd,amsthm,amssymb,amsfonts}
\usepackage{mathrsfs}
\usepackage{dsfont}
\usepackage{stmaryrd}
\usepackage{euscript}
\usepackage{expdlist}
\usepackage{enumerate}

\input xy

\newtheorem{thm}{Theorem}[section]
\newtheorem{thmA}{Theorem}

\newtheorem*{thm*}{Theorem}

\newtheorem{lm}[thm]{Lemma}
\newtheorem{cor}[thm]{Corollary}
\newtheorem*{cor*}{Corollary}
\newtheorem{prop}[thm]{Proposition}
\newtheorem*{conj*}{Conjecture}

\theoremstyle{Remark}

\theoremstyle{definition}
\newtheorem*{defn*}{Definition}
\newtheorem{Remark}[thm]{Remark}
\newtheorem{I_Remark*}{Remark}

\newcommand{\nc}{\newcommand}

\newcommand{\beq}{\begin{equation}}
\newcommand{\eeq}{\end{equation}}
\newcommand{\bpmx}{\begin{pmatrix}}
\newcommand{\epmx}{\end{pmatrix}}
\newcommand{\bbmx}{\begin{bmatrix}}
\newcommand{\ebmx}{\end{bmatrix}}

\def\parref#1{\ref{#1}}
\def\thmref#1{Theorem~\parref{#1}}

\def\propref#1{Proposition~\parref{#1}}
\def\corref#1{Corollary~\parref{#1}}

\def\lmref#1{Lemma~\parref{#1}}

\def\makeop#1{\expandafter\def\csname#1\endcsname
  {\mathop{\rm #1}\nolimits}\ignorespaces}

\makeop{Hom}   \makeop{End}   \makeop{Aut}   
\makeop{Pic} \makeop{Gal}       \makeop{Div} \makeop{Lie}
\makeop{PGL}   \makeop{Corr} \makeop{PSL} \makeop{sgn} \makeop{Spf}
 \makeop{Tr} \makeop{Nr} \makeop{Fr} \makeop{disc}
\makeop{Proj} \makeop{supp} \makeop{ker}   \makeop{Im} \makeop{dom}
\makeop{coker} \makeop{Stab} \makeop{SO} \makeop{SL} \makeop{SL}
\makeop{Cl}    \makeop{cond} \makeop{Br} \makeop{inv} \makeop{rank}
\makeop{id}    \makeop{Fil} \makeop{Frac}  \makeop{GL} \makeop{SU}
\makeop{Trd}   \makeop{Sp} \makeop{Tr}    \makeop{Trd} \makeop{Res}
\makeop{ind} \makeop{depth} \makeop{Tr} \makeop{st} \makeop{Ad}
\makeop{Int} \makeop{tr}    \makeop{Sym} \makeop{can} \makeop{SO}
\makeop{torsion} \makeop{GSp} \makeop{Tor}\makeop{Ker} \makeop{rec}
\makeop{Ind} \makeop{Coker}
 \makeop{vol} \makeop{Ext} \makeop{gr} \makeop{ad}
 \makeop{Gr}\makeop{corank} \makeop{Ann}
\makeop{Hol} 
\makeop{Fitt} \makeop{Mp} \makeop{CAP}

\def\makebb#1{\expandafter\def
  \csname bb#1\endcsname{{\mathbb{#1}}}\ignorespaces}
\def\makebf#1{\expandafter\def\csname bf#1\endcsname{{\bf
      #1}}\ignorespaces}
\def\makegr#1{\expandafter\def
  \csname gr#1\endcsname{{\mathfrak{#1}}}\ignorespaces}
\def\makescr#1{\expandafter\def
  \csname scr#1\endcsname{{\EuScript{#1}}}\ignorespaces}
\def\makecal#1{\expandafter\def\csname cal#1\endcsname{{\mathcal
      #1}}\ignorespaces}

\def\doLetters#1{#1A #1B #1C #1D #1E #1F #1G #1H #1I #1J #1K #1L #1M
                 #1N #1O #1P #1Q #1R #1S #1T #1U #1V #1W #1X #1Y #1Z}
\def\doletters#1{#1a #1b #1c #1d #1e #1f #1g #1h #1i #1j #1k #1l #1m
                 #1n #1o #1p #1q #1r #1s #1t #1u #1v #1w #1x #1y #1z}
\doLetters\makebb   \doLetters\makecal  \doLetters\makebf
\doLetters\makescr
\doletters\makebf   \doLetters\makegr   \doletters\makegr

    \def\setminus{\smallsetminus}

\normalsize

\makeop{Ram} \makeop{Rep} \makeop{mass}

\makeop{Bl}

\def\diag#1{\mathrm{diag}(#1)}
\def\el{\ell}

\def\cS{{\mathcal S}}

\def\cP{{\mathcal P}}

\def\bbL{\mathbb L}

\def\ot{\otimes}

\def\hookto{\hookrightarrow}
\def\longto{\longrightarrow}
  \nc{\opp}{\mathrm{opp}} \nc{\ul}{\underline}

\def\XYmatrix{\xymatrix@M=8pt} 
\def\ncmd{\newcommand}
\ncmd{\xysubset}[1][r]{\ar@<-2.5pt>@{^(-}[#1]\ar@<2.5pt>@{_(-}[#1]}
\ncmd{\XYmatrixc}[1]{\vcenter{\XYmatrix{#1}}}
\ncmd{\xyto}[1][r]{\ar@{->}[#1]}
\ncmd{\xyinj}[1][r]{\ar@{^(->}[#1]}
\ncmd{\xysurj}[1][r]{\ar@{->>}[#1]}
\ncmd{\xyline}[1][r]{\ar@{-}[#1]}
\ncmd{\xydotsto}[1][r]{\ar@{.>}[#1]}
\ncmd{\xydots}[1][r]{\ar@{.}[#1]}
\ncmd{\xyleadsto}[1][r]{\ar@{~>}[#1]}
\ncmd{\xyeq}[1][r]{\ar@{=}[#1]} \ncmd{\xyequal}[1][r]{\ar@{=}[#1]}
\ncmd{\xyequals}[1][r]{\ar@{=}[#1]}
\ncmd{\xymapsto}[1][r]{l\ar@{|->}[#1]}\ncmd{\xyimplies}[1][r]{\ar@{=>}[#1]}
\ncmd{\xyiso}{\ar[r]_-{\sim}}
\def\injxy{\ar@{^(->}}

\newcommand{\pMX}[4]{\begin{pmatrix}
{#1}& {#2}\\
{#3}&{#4}\end{pmatrix} }

\newcommand{\seesaw}[4]{{#1}\ar@{-}[rd]\ar@{-}[d]&{#2}\ar@{-}[d]\\
{#3}\ar@{-}[ru]&{#4}}

\def\x{{\times}}

\def\e{\varepsilon} 

\newcommand\stt[1]{\left\{#1\right\}}
\def\ep{\epsilon}

\renewcommand\pmod[1]{\,(\mbox{mod }{#1})}

\renewcommand\Re{\text{Re}\,}

\usepackage{color}
\usepackage{hyperref}
\usepackage{MnSymbol}
\usepackage{mathdots}
\usepackage{tikz-cd}

\def\SO{{\rm SO}}
\def\M{{\rm Mat}}

\def\a{\alpha}
\def\b{\bar}

\title{Local newforms for generic representations of $p$-adic ${\rm SO}_{2n+1}$: Uniqueness}
\author{Yao Cheng}
\date{\today}
\address{No. 151, Yingzhuan Road, Tamsui District, New Taipei City 251, Taiwan (R.O.C),  Lui-Hsien Memorial Science Hall.}
\email{briancheng@o365.tku.edu.tw}

\begin{document}
\maketitle
\begin{abstract}
The conjectural theory of local newofmrs for the split $p$-adic group $\SO_{2n+1}$, proposed by Gross, predicts that the space of 
local newforms in a generic representation is one-dimensional. In this note, we prove that this space is at most one-dimensional and 
verify its expected arithmetic properties, conditional on existence. These results play an important role in our proof of the existence 
part of the newform conjecture.
\end{abstract}

\section{Introduction}
In the 1970s, Casselman (\cite{Casselman1973}) developed the theory of local newforms for generic representations of 
${\rm GL}_2$ over a $p$-adic field $F$.  As the name suggests, local newforms are precisely the ``local components" of 
modular newforms--a theory developed by Atkin and Lehner (\cite{AtkinLehner1970}) around the same time. 
Owing to this connection, local newforms inherit significant arithmetic properties (see \cite{Schmidt2002}). For instance, the 
non-vanishing of the first Fourier coefficient of a modular newform reflects the non-vanishing of Whittaker functionals on its local 
components. Likewise, the eigenvalue of the Atkin–Lehner involution on a modular newform factors as the product of the eigenvalues 
of the corresponding local involutions.

Casselman's results were subsequently extended to generic representations of $\GL_r(F)$ by Jacquet--Piatetski-Shapiro--Shalika
(\cite{JPSS1983}, see also \cite{Jacquet2012}, \cite{Matringe2013}), and other classical groups by various authors. 
In particular, Roberts and Schmidt (\cite{RobertsSchmidt2007}) developed the theory of local newforms for generic representations of 
${\rm GSp}_4(F)$ with trivial central character. These local newforms also exhibit arithmetic properties analogous to those 
established by Casselman.

By the accidental isomorphisms ${\rm PGL}_2\simeq{\rm SO}_3$ and ${\rm PGSp}_4\simeq{\rm SO}_5$, the results of Casselman 
and Roberts-Shcmidt can be placed into a single framework. Building on these results, Gross (\cite{Gross2015}) proposed a 
conjectural theory of local newforms for generic representations of the split group ${\rm SO}_{2n+1}(F)$. 
This conjecture asserts that the space of local newforms is one-dimensional and that local newforms possess arithmetic properties.

The aim of this note is to prove the uniqueness part of this conjecture; namely, that the space of local newforms is at most 
one-dimensional, and verify its expected arithmetic properties, conditional on existence. 

\subsection{Newform conjecture}
Let $V_n$ (with $n\ge 1$) be the $(2n+1)$-dimensional quadratic space over $F$ whose discriminant and Hasse invariant
are both equal to $1$. Let $\SO(V_n)\simeq\SO_{2n+1}(F)\subset{\rm SL}_{2n+1}(F)$ denote the associated split special orthogonal 
group, and let $U_n\subset\SO_{2n+1}(F)$ be its maximal unipotent subgroup consisting of upper triangular matrices.
Let $\psi$ be an unramified additive character of $F$.  Define the non-degenerate character $\psi_{U_n}:U_n\longto\bbC^{\x}$ by 
\[
\psi_{U_n}(u)=\psi\left(u_{1,2}+\cdots+u_{n-1,n}+2^{-1}u_{n,n+1}\right)
\quad\text{for $u=(u_{i,j})\in U_n$.}
\]

Let $\pi$ be an irreducible generic representation of $\SO_{2n+1}(F)$, i.e. ${\rm Hom}_{U_n}\left(\pi,\psi_{U_n}\right)\ne 0$, 
with associated $L$-parameter $\phi_\pi$ (see \cite{JiangSoudry2003}, \cite{JiangSoudry2004}, \cite{Arthur2013}). 
The $\epsilon$-factor $\epsilon(s,\phi_\pi,\psi)$ attached to $\phi_\pi$, $\psi$ and the standard 
representation of the $L$-group of $\SO_{2n+1}(F)$ (see \cite{Tate1979}), can be expressed as
\[
\epsilon(s,\phi_\pi,\psi)
= 
\e_\pi q^{-c_\pi\left(s-\frac{1}{2}\right)},
\]
for some $\e_\pi=\pm 1$ and integer $c_\pi\ge 0$, where $q$ denotes the cardinality of the residue field of $F$.

In \cite{Gross2015} (see also \cite{Tsai2013}, \cite{Tsai2016}), Gross defined a family $\stt{K_{n,m}}_{m\ge 0}$ of open compact 
subgroups of $\SO_{2n+1}(F)$ (see \S\ref{S:K}), generalizing the families introduced by Casselman and Roberts-Schmidt to
arbitrary $n$. He also defined a family $\stt{J_{n,m}}_{m\ge 0}$ of open compact subgroups of $\SO_{2n+1}(F)$ such that 
$K_{n,0}=J_{n,0}$ and, for each $m\ge 1$, $K_{n,m}$ is a normal subgroup of $J_{n,m}$ of index $2$. In particular, the subgroup 
$J_{n,m}$ acts naturally on the subspace $\pi^{K_{n,m}}$ of $K_{n,m}$-fixed vectors of $\pi$. Here and below, we often abuse notation 
by writing $\pi$ for its underlying space.

Now, we can state the following conjecture due to Gross:
\begin{conj*}
Let $\pi$ be an irreducible generic representation of $\SO_{2n+1}(F)$. Then 
\begin{itemize}
\item[(1)] the subspaces satisfy $\pi^{K_{n,m}}=0$ for $0\le m<c_\pi$ and $\dim_\bbC\pi^{K_{n,c_\pi}}=1$;
\item[(2)] the action of $J_{n,c_\pi}/K_{n,c_\pi}$ on $\pi^{K_{n,c_\pi}}$ is given by the scalar $\e_\pi$;
\item[(3)] the natural pairing of the one-dimensional spaces 
\[
{\rm Hom}_{K_{n,c_\pi}}\left(1,\pi\right)\,\x\,{\rm Hom}_{U_n}\left(\pi,\psi_{U_n}\right)\longto\bbC,
\]
is non-degenerate.
\end{itemize}
\end{conj*}

\begin{Remark}\noindent
\begin{itemize}
\item[(1)]
As already noted, this conjecture holds for $n=1,2$. For general $n$, the conjecture holds when $\pi$ is unramified. 
Moreover, when $\pi$ is supercuspidal, this conjecture was treated in Tsai's PhD thesis (\cite{Tsai2013}). 
\item[(2)]
Following the terminology in the literature, we call $\pi^{K_{n,c_\pi}}$ the 
space of newforms of $\pi$; on the other hand, the spaces\footnote{When $n=1$, the inclusion $K_{1,m+1}\subset K_{1,m}$ implies 
$\pi^{K_{1,m}}\subset\pi^{K_{1,m+1}}$ for every $m\ge 0$. Thus in this case, oldforms are vectors in $\pi^{K_{1,m}}$ for $m>c_\pi$
and are not new. On the other hand, when $n\ge 2$, such inclusions no longer hold.} $\pi^{K_{n,m}}$ for $m>c_\pi$ are called the 
spaces of oldforms of $\pi$. In \cite[Conjecture 9.1.10]{Tsai2013}, Tsai also proposed conjectural dimension formulas for the spaces 
of oldforms.
\end{itemize}
\end{Remark}

\subsection{Main result}
As mentioned, this note proves that the space of newforms is at most one-dimensional and establishes its expected arithmetic
properties. More concretely, we prove the following.

\begin{thmA}\label{T:main}
Let $\pi$ be an irreducible generic representation of $\SO_{2n+1}(F)$. Then 
\begin{itemize}
\item[(1)] the subspaces satisfy $\pi^{K_{n,m}}=0$ for $0\le m<c_\pi$ and $\dim_\bbC\pi^{K_{n,c_\pi}}\le 1$;
\item[(2)] if $\pi^{K_{n,c_\pi}}\ne 0$, then the action of $J_{n,c_\pi}/K_{n,c_\pi}$ on $\pi^{K_{n,c_\pi}}$ is given by the scalar $\e_\pi$;
\item[(3)] if $\pi$ is tempered and $\pi^{K_{n,c_\pi}}\ne 0$, then the natural pairing of the one-dimensional spaces 
\[
{\rm Hom}_{K_{n,c_\pi}}\left(1,\pi\right)\,\x\,{\rm Hom}_{U_n}\left(\pi,\psi_{U_n}\right)\longto\bbC,
\]
is non-degenerate.
\end{itemize}
\end{thmA}

\begin{Remark}
In our previous paper \cite{YCheng2022}, we proved \thmref{T:main} $(2)$ under the hypothesis that $\dim_\bbC \pi^{K_{n,c_\pi}}=1$
and that the Whittaker functional is nontrivial on $\pi^{K_{n,c_\pi}}$. We also showed that, under the same hypothesis, the dimensions 
of the spaces of oldforms are greater than or equal to those predicted by Tsai in \cite[Conjecture 9.1.10]{Tsai2013}.

On the other hand, using the results of this note, one can show that if the dimensions of the spaces of oldforms predicted by Tsai
are valid for tempered representations, then they are also valid for generic representations.
\end{Remark}

\subsection{Ingredients of the proof}
The proof of \thmref{T:main} relies on the following three ingredients:
\begin{itemize}
\item double coset decompositions;
\item local Rankin-Selberg integrals (\cite{Ginzburg1990}, \cite{Soudry1993}) for $\SO_{2n+1}\x\GL_r$ with $1\le r\le n$;
\item uniqueness of local Gross-Prasad periods (\cite{AGRS2010}) for the pair $(\SO_{2n+1}, \SO_{2n})$.
\end{itemize}
Using double coset decompositions, the proof of \thmref{T:main} $(1)$ reduces to the tempered case. To prove \thmref{T:main} for 
tempered representations, we apply the maps $\Xi_{r,m}$ (with $1\le r\le n$) on $\pi^{K_{n,m}}$, constructed from the 
local Rankin-Selberg integrals for $\SO_{2n+1}\x\GL_r$ in \cite[Proposition 6.7]{YCheng2022}.  A similar proof already appears in 
\cite[\S 7.1]{YCheng2022}, under the hypothesis indicated in the previous remark. To remove the hypothesis in loc. cit., we must  
show that the maps $\Xi_{n,m}$ are injective for all $m$ when $\pi$ is tempered. To this end, we apply the third ingredient, namely, 
the uniqueness of local Gross-Prasad periods for the pair $(\SO_{2n+1}, \SO_{2n})$ (see \lmref{L:inj}).

\subsection{An outline of this note}
In \S\ref{S:SO}, we introduce the special orthogonal groups considered in this note and describe their structure.

In \S\ref{S:K}, we define the open compact subgroups introduced by Gross \cite{Gross2015} and obtain useful decompositions of 
these subgroups in \lmref{L:K decomp}.

In \S\ref{S:DCD}, we study the double coset decompositions of $\SO_{2n+1}$ with respect to its maximal parabolic subgroups and 
$K_{n,m}$. The main result here is \propref{P:set of rep}.

In \S\ref{S:gp from int}, we analyze the intersections of conjugates of $K_{n,m}$ with maximal parabolic subgroups, building on 
\propref{P:set of rep}. The main result of this section is \propref{P:gp of int K^0}, which plays a key role in the reduction to the 
tempered case.

In \S\ref{S:proof}, we prove \thmref{T:main}. More specifically, in \S\ref{SS:reduction}, we reduce the proof to the tempered
case using \lmref{L:reduction}. In \S\ref{SS:RS int}, we briefly review the local Rankin–Selberg integrals for
$\SO_{2n+1}\times\GL_r$ with $1\le r\le n$ and establish the key \lmref{L:inj}. We also recall 
\cite[Proposition 6.7]{YCheng2022} in \propref{P:key}. Finally, in \S\ref{SS:proof}, we prove the tempered case, thereby completing 
the proof.

\subsection{Notation and conventions}
Let $F$ be a finite extension of $\mathbb{Q}_p$. Denote by $\frak{o}$ the valuation ring of $F$, by $\frak{p}$ its maximal ideal, by 
$\varpi$ a uniformizer of $\frak{p}$, and by $\frak{f}=\frak{o}/\frak{p}$ the residue field of $F$, of cardinality $q$.
Let $|\cdot|_F$ be the absolute value on $F$, normalized so that $|\varpi|_F=q^{-1}$. For an integer $m\ge 0$, set 
$\frak{u}^0=\frak{o}^\times$ and $\frak{u}^m=1+\frak{p}^m$ for $m\ge 1$.
Let $\M_{m,n}(F)$ denote the space of $m\times n$ matrices with entries in $F$. For $1\le i,j\le n$, let
$E^{m,n}_{i,j}\in\M_{m,n}(F)$ be the matrix with a single nonzero entry $1$ in the $(i,j)$-th position. When $m=n$, we also write
$\M_n(F)$ for $\M_{n,n}(F)$, and $E^n_{i,j}$ for $E^{n,n}_{i,j}$. The identity element in $\M_n(F)$ is denoted by
$I_n=\sum_{i=1}^n E^n_{i,i}$. Define $J_n\in\M_n(F)$ by $J_n=\sum_{j=1}^n E^n_{1,n+1-j}$. For $a\in\GL_n(F)$, set
$a^* = J_n {}^t a^{-1} J_n$.
Fix an additive character $\psi$ of $F$ that is trivial on $\frak{o}$ but non-trivial on $\frak{p}^{-1}$.
In this note, a representation of an $\ell$-group means a smooth complex representation of finite length.

\section{Special orthogonal groups}\label{S:SO}

\subsection{Quadratic spaces}
Let $n\ge 0$ be an integer and $V_n$ be an $(2n+1)$-dimensional $F$-linear space 
equipped with a non-degenerate symmetric bilinear form $\langle\cdot,\cdot\rangle$.
We assume that $V_n$ admits an ordered basis $\gamma_n=\stt{e_{-n},\ldots, e_{-1}, e_0,  e_{1},\ldots, e_{n}}$ satisfying
\[
\langle e_0\,,\, e_0\rangle=2
 \quad\text{and}\quad
\langle e_i\,,\, e_j\rangle=\langle e_{-i}\,,\, e_{-j}\rangle=0,
\quad
\langle e_i\,,\, e_{-j}\rangle=\delta_{i,j},
\]
for $1\le i,j\le n$. 
Thus, the Gram matrix of $\langle\cdot,\cdot\rangle$ associated with $\gamma_n$ is given by
\[
S_n
=
\begin{pmatrix}
&&J_n\\
&2\\
J_n
\end{pmatrix}.
\]
\subsection{Special orthogonal groups}
The special orthogonal group $\SO(V_n)$ associated with $\left(V_n,\langle\cdot,\cdot\rangle\right)$ is defined by 
\[
\SO(V_n)=\stt{h\in{\rm SL}(V_n)\mid\text{$\langle gu\,,\,gv\rangle=\langle u\,,\,v\rangle$ for all $u,v\in V_n$}}.
\]
Using the ordered basis $\gamma_n$, it can be realized as a matrix group
\[
\SO_{2n+1}(F)
=
\stt{g\in{\rm SL}_{2n+1}(F)\mid {}^tg\,S_n\,g=S_n}.
\]
We view $\SO_{2n+1}$ as a (split) algebraic group defined over $F$. 

Let $B_n=T_nU_n\subset\SO_{2n+1}$ be the upper triangular Borel subgroup with the unipotent radical $U_n$, where 
\[
T_n
=
\stt{
t=\diag{t_1, \ldots, t_n, 1, t^{-1}_n, \ldots, t_1^{-1}}\mid t_1,\ldots, t_n\in\mathbb{G}_m
}
\]
is the diagonal torus. 

\subsection{Roots and co-roots}
The (additive) character group $X^*(T_n)={\rm Hom}(T_n, \mathbb{G}_m)$ admits a standard basis $\ep_1, \ldots, \ep_n$ defined by 
\[
\ep_i(t)=t_i\quad\text{for $1\le i\le n$.} 
\]
The root system $R_n$ of $\SO_{2n+1}$ is then given by
\[
R_n
=
\stt{\pm \ep_i\pm\ep_j\mid 1\le i< j\le n}
\cup
\stt{\pm \ep_\ell\mid 1\le \ell\le n}.
\]
The choice of $B_n$ determines subsets $\Delta_n\subset R_n^+\subset R_n$ of positive roots and simple roots:
\[
\Delta_n
=
\stt{\ep_1-\ep_2,\ldots, \ep_{n-1}-\ep_n, \ep_n}
\subset
R_n^+
=
\stt{\ep_i\pm\ep_j\mid 1\le i<j\le n}
\cup
\stt{\ep_\ell\mid 1\le \el\le n}.
\]
We denote $\a_r=\ep_1-\ep_r$ for $1\le r\le n-1$, and $\a_n=\ep_n$.

Let $\ep^*_1,\ldots, \ep_n^*$ be the basis of the co-character group $X_*(T_n)={\rm Hom}(\bbG_m, T_n)$ that is dual to 
$\ep_1,\ldots,\ep_n$ with respect to the natural pairing $(\cdot,\cdot)$ between $X^*(T_n)$ and $X_*(T_n)$. 
For an integer $1\le r\le n$, we define
\begin{equation}\label{E:lambda_r}
\lambda_r=\ep^*_1+\cdots+\ep^*_r.
\end{equation}
Given $z\in F^\x$ and $\mu\in X_*(T_n)$, we also write $z^\mu$ for $\mu(z)\in T_n(F)$.

\subsection{Root elements}
Given an integer $1\le i\le n$, we define $i^*=2n+2-i$. Note the relations
\[
i+i^*=2n+2
\quad\text{and}\quad
\left(i^*\right)^*=i. 
\]

Let $\alpha\in R_n$ be a root. Define the associated root element $x_\alpha(y)\in\SO_{2n+1}(F)$ for 
$y\in F$ by
\[
x_{\ep_i- \ep_j}(y)
=
I_{2n+1}+yE^{2n+1}_{i,j}-yE^{2n+1}_{j^*, i^*},
\]
\[
x_{\ep_i+\ep_j}(y)
=
I_{2n+1}+yE^{2n+1}_{i,j^*}-yE^{2n+1}_{j,i^*},
\]
\[
x_{\ep_\el}(y)
=
I_{2n+1}-2yE^{2n+1}_{\el,n+1}+ yE^{2n+1}_{n+1,\el^*}-y^2E^{2n+1}_{\el,\el^*},
\]
\[
x_{-\ep_\el}(y)
=
I_{2n+1}+2yE^{2n+1}_{\el^*,n+1}- yE^{2n+1}_{n+1,\el}-y^2E^{2n+1}_{\el^*,\el},
\]
and
\[
x_{-(\ep_i\pm\ep_j)}(y)={}^t x_{\ep_i\pm\ep_j}(y),
\]
for $1\le i<j\le n$, and $1\le \el\le n$. In general, we define
\[
x_\alpha(Y)=\stt{x_\alpha(y)\mid y\in Y},
\]
where $Y$ is a subset of $F$.

Note the identity
\begin{equation}\label{E:z^mu conjugate}
z^\mu\,x_{\a}(y)\, z^{-\mu}
=
x_\a\left( z^{(\a,\mu)}y\right),
\end{equation}
for $z\in F^\x$, $y\in F$, $\mu\in X_*(T_n)$ and $\a\in R_n$.

\subsection{Weyl elements}\label{SS:Weyl}
The Weyl group $W_n\subset{\rm Aut}(\bbR^n)$ of $R_n$ is isomorphic to $\frak{S}_n\ltimes\stt{\pm 1}^n$, and can be 
identified with $N_{\SO_{2n+1}(F)}(T_n(F))/T_n(F)$ as follow. Given $\alpha\in R_n$ and $y\in F^\x$, define
\begin{equation}\label{E:general Weyl elt}
w_\alpha(y)
=
x_\alpha(y)\,x_{-\alpha}\left(-y^{-1}\right)\,x_\alpha(y).
\end{equation}
Then the reflection $S_\alpha:\mathbb{R}^n\to\mathbb{R}^n$ associated with $\alpha$ is identified with $w_\alpha(1)$ 
modulo $T_n(F)$. 

For simplicity, we also denote $w_{\a}=w_{\a}(1)$ for $\a\in R_n$. On the other hand, for $1\le\el\le n$ and $k\in\bbZ$, we define 
$w_{\ep_\el, k}=\varpi^{-k\ep^*_\el}\,w_{\ep_\el}$. We have

\begin{itemize}
\item
$w_{\ep_i-\ep_j}$ is the matrix obtained from $I_{2n+1}$ by interchanging the $i$-th and $j$-th rows, as well as $i^*$-th and 
$j^*$-th rows, and then multiplying the $j$-th and $j^*$-th rows by $-1$; 
\item
$w_{\ep_i+\ep_j}$  is the matrix obtained from $I_{2n+1}$ by interchanging the $i$-th and $j^*$-th rows, as well as 
$j$-th and $i^*$-th rows, and then multiplying the $j$-th and $i^*$-th rows by $-1$;
\item
$w_{\ep_\ell}$  is the matrix obtained from $I_{2n+1}$ by interchanging the $\el$-th and $\el^*$-th rows, and then
multiplying the $\ell$-th, $(n+1)$-th and $\el^*$-th rows by $-1$;
\item
 $w_{-\alpha}={}^tw_\alpha$ for $\alpha\in R_n^+$.
\end{itemize}

Note the relation
\begin{equation}\label{E:w conjugate}
w_\alpha\,x_\beta(y)\,w_\alpha^{-1}=x_{S_\alpha(\beta)}(y),
\end{equation}
for $\alpha,\beta\in R_n$. 

\subsection{Parabolic subgroups}\label{SS:para}
Let $S\subset\Delta_n$ be a subset (possibly empty). Define a subset $R_S$ of $R_n$ by
\[
R_S
=
\stt{\sum_{\a\in\Delta_n}n_\a \a\in R_n\mid \text{$n_\a\in\mathbb{Z}$ with $n_\a\ge 0$ when $\a\in S$}}.
\]
Let $P_S=M_SN_S\supset B_n$ be the parabolic subgroup of $\SO_{2n+1}$ defined by 
\[
P_S(F)
=
\langle
T_n(F),\, x_\a(F)\mid \a\in R_S
\rangle
\]
where $N_S$ is the unipotent radical of $P_S$ and $M_S:=N_{P_S}(T_n)$ is a Levi subgroup of $P_S$. Note that $B_n=P_{\Delta_n}$ 
and $\SO_{2n+1}=P_\emptyset$ according to our definition.

For simplicity, we write $P_\a=P_{\stt{\a}}$ for $\a\in\Delta_n$. These are maximal parabolic subgroups of $\SO_{2n+1}$. Note that
\[
M_{\ep_1-\ep_r}(F)
=
\stt{
\diag{a,g_0,a^*}\mid
a\in\GL_r(F),\,g_0\in\SO_{2(n-r)+1}(F)
}
\cong
\GL_r(F)\,\x\,\SO_{2(n-r)+1}(F)
\]
for $1\le r\le n-1$, and 
\[
M_{\ep_n}(F)
=
\stt{\diag{a,a^*}\mid a\in\GL_n(F)}\cong \GL_n(F).
\]
These parabolic subgroups play an important role in this note.

\subsection{Subgroups and embeddings}
Let $1\le r\le n$ be an integer and 
\[
\SO_{2r}(F)
=
\stt{h\in{\rm SL}_{2r}(F)\mid {}^t\,h\, J_{2r} h=J_{2r}}
\]
be an even special orthogonal group of rank $r$. We consider $\SO_{2r}$ as a (split) algebraic group defined over $F$, and
identify it as a subgroup of $\SO_{2n+1}(F)$ via the embedding 
\begin{equation}\label{E:embedding}
\SO_{2r}(F)\ni\pMX{a}{b}{c}{d}\longmapsto
\begin{pmatrix}
a&&b\\
&I_{2(n-r)+1}\\
c&&d
\end{pmatrix}\in \SO_{2n+1}(F),
\end{equation}
where $a,b,c,d\in{\rm Mat}_r(F)$. In other words, $\SO_{2r}(F)$ is isomorphic to the subgroup of $\SO(V_n)$ fixing the vectors 
$e_{-n+r},\ldots, e_{-1}, e_0, e_1,\ldots, e_{n-r}$.

The Weyl group of the root system of $\SO_{2n}$ is isomorphic to $\frak{S}_n\ltimes\stt{\pm 1}^{n-1}$, can be identified with 
$N_{\SO_{2n}(F)}(T_n(F))/T_n(F)$, and satisfies
\begin{equation}\label{E:Weyl gp quo}
\left(N_{\SO_{2n+1}(F)}(T_n(F))/T_n(F)\right)
/\left(N_{\SO_{2n}(F)}(T_n(F))/T_n(F)\right)
=
\stt{T_n(F),\,w_{\ep_j}\,T_n(F)},
\end{equation}
for any $1\le j\le n$.

\section{Open compact subgroups}\label{S:K}
In this section, we introduce the open compact subgroups $K_{n,m}$ of $\SO_{2n+1}(F)$ defined in \cite{Gross2015}, and examine their 
properties. We provide two descriptions, each offering distinct advantages for understanding these subgroups.

Let $m\ge 0$ be an integer and $\bbL_{n,m}\subset V_n$ be an $\frak{o}$-lattice defined by
\[
\bbL_{n,m}
=
\frak{o}e_{-n}\oplus\cdots\oplus\frak{o}e_{-1}\oplus\frak{p}^m e_0\oplus\frak{p}^m e_1\oplus\cdots\oplus\frak{p}^m e_n.
\]
Let $J_{n,m}\subset \SO_{2n+1}(F)$ be the open compact subgroup that stabilizes the lattice $\bbL_{n,m}$. Then 
$K_{n,0}:=J_{n,0}$ is the hyperspecial maximal compact subgroup of $\SO_{2n+1}(F)$. For $m\ge 1$, $K_{n,m}$ is a normal 
subgroup of $J_{n,m}$ of index $2$, defined as follows.

\subsection{First description}
Suppose that $m\ge 1$. The group $J_{n,m}$ is isomorphic to the group of $\frak{o}$ points of a group scheme over $\frak{o}$. 
More precisely, if we equip $\bbL_{n,m}$ with the symmetry bilinear form 
$\langle\cdot,\cdot\rangle_m:=\varpi^{-m}\langle\cdot,\cdot\rangle$, 
then the Gram matrix of $(\bbL_{n,m}, \langle\cdot,\cdot\rangle_m)$ associated to the ordered basis 
$\stt{e_{-n},\ldots, e_{-1}, \varpi^m e_0, \varpi^m e_1,\ldots, \varpi^m e_n}$
is 
\[
S_{n,m}
=
\begin{pmatrix}
&&J_n\\
&2\varpi^m&\\
J_n&&
\end{pmatrix},
\]
and 
\[
\tilde{J}_{n,m}:=
\stt{g=(g_{ij})\in{\rm SL}_{2n+1}(\frak{o})\mid
\text{$^tg\,S_{n,m}\,g=S_{n,m}$ and $g_{j,n+1}\in\frak{p}^m$ for $1\leq j\neq n+1\leq 2n+1$}}
\]
is the group of $\frak{o}$ point of a group scheme $\tilde{\bf{J}}_{n,m}$ over $\frak{o}$, i.e. 
$\tilde{J}_{n,m}=\tilde{\bf{J}}_{n,m}(\frak{o})$ (see \cite[Theorem 3.6]{Shahabi2018}). Now we have 
\[
J_{n,m}=t_m^{-1}\,\tilde{J}_{n,m}\,t_{m}
\]
where $t_m:=\diag{\varpi^mI_n,1,I_n}$ with conjugation taken in $\GL_{2n+1}(F)$.

The reduction modulo $\frak{p}$ map gives rise to a surjective homomorphism
\[
\tilde{J}_{n,m}=\tilde{\bf{J}}_{n,m}(\frak{o})
\relbar\joinrel\twoheadrightarrow
{\rm S}({\rm O}_{2n}(\frak{f})\x{\rm O}_1(\frak{f}))
\relbar\joinrel\twoheadrightarrow
{\rm O}_{2n}(\frak{f})
\overset{{\rm det}}{\relbar\joinrel\twoheadrightarrow}
\stt{\pm 1}.
\]
Then $K_{n,m}\subset J_{n,m}$ is defined to be the index $2$ normal subgroup such that 
$t_m K_{n,m} t_m^{-1}\subset \tilde{J}_{n,m}$ is the kernel of the above homomorphism. 

\subsection{Second description}
The reference for this subsection is Shahabi's PhD thesis \cite{Shahabi2018}. Suppose again that $m\ge 1$. 
Then $J_{n,m}$ is given by 
\[
J_{n,m} 
=
\bordermatrix{
              & n     & 1    & n     \cr
    n     & \frak{o} &\frak{o} &\frak{p}^{-m}     \cr
    1     &\frak{p}^m& \frak{o} &\frak{o}\cr
    n     &\frak{p}^m&\frak{p}^m& \frak{o}   \cr
            }\cap \SO_{2n+1}(F).
\]
Now define 
\[
K_{n,m}=\stt{k=(k_{i,j})\in J_{n,m}\mid k_{n+1,n+1}\in 1+\frak{p}}.
\]

Evidently, the second description is more straightforward. However, it does not make it immediately clear that $K_{n,m}$ is a subgroup of 
$J_{n,m}$ with index $2$. Note that 
\begin{equation}\label{E:AT elt}
w_{\ep_j,m}
\in
J_{n,m}\setminus K_{n,m}
\end{equation}
for any $1\le j\le n$. Following convention in the literature, the elements $w_{\ep_1,m}$ are called the $Atkin$-$Lehner$ elements.


\subsection{Decompositions of $K_{n,m}$}
Recall that $\SO_{2r}(F)\subset\SO_{2n+1}(F)$ via \eqref{E:embedding} for $r\le n$. Define 
\begin{equation}\label{E:H}
H_{r,m}=\SO_{2r}(F)\cap K_{n,m}.
\end{equation}
Then $H_{r,0}$ and $H_{r,1}$ are two non-conjugate hyperspecial maximal compact subgroup of $\SO_{2r}(F)$, and we have
\begin{equation}\label{E:H_r,e}
H_{r,e}
=
\langle T_r(\frak{o}),\,x_{\pm(\ep_i-\ep_j)}(\frak{o}), x_{\ep_i+\ep_j}(\frak{p}^{-e}),\,x_{-(\ep_i+\ep_j)}(\frak{p}^e)\mid 1\le i<j\le r\rangle
\end{equation}
for $e=0,1$, where $T_r(\frak{o})=T_n(F)\cap H_{r,0}$. For arbitrary $m$, the following relation holds:
\begin{equation}\label{E:H_r,m}
H_{r,m}=\varpi^{-\lfloor\frac{m}{2}\rfloor\lambda_r}\,H_{r,e}\,\varpi^{\lfloor\frac{m}{2}\rfloor\lambda_r}
\end{equation}
where $\lambda_r\in X_*(T_n)$ is given by \eqref{E:lambda_r}, and $e=0,1$ satisfies $m\equiv e\pmod{2}$.

We now state and prove a lemma due to Tsai (cf.\,\cite[Proposition 7.1.3]{Tsai2013}). For convenience, we introduce the elements 
\begin{equation}\label{E:w_r,m}
w_{r,m}
=
\prod_{j=1}^r w_{\ep_j, m}\in J_{n,m}
\end{equation}
for $1\le r\le n$ and $m\ge 0$.

\begin{lm}\label{L:K decomp}
Suppose that $m\ge 1$. Then 
\begin{align*}
K_{n,m}
&=
\prod_{j=1}^n x_{\ep_j}(\frak{o})\,\prod_{j=1}^n x_{-\ep_j}(\frak{p}^m)\, H_{n,m}\\
&=
\prod_{j=1}^n x_{-\ep_j}(\frak{p}^m)\,\prod_{j=1}^n x_{\ep_j}(\frak{o})\, H_{n,m}.
\end{align*}
\end{lm}

\begin{proof}
It suffices to verify the first identity. Indeed, suppose we have 
\begin{equation}\label{E:K decomp}
K_{n,m}=\prod_{j=1}^n x_{\ep_j}(\frak{o})\,\prod_{j=1}^n x_{-\ep_j}(\frak{p}^m)\, H_{n,m}.
\end{equation}
Then since $w_{n,m}$ normalizes both $K_{n,m}$ and $\SO_{2n}(F)$ and
\[
w_{n,m}\, x_{\ep_j}(\frak{o})\, w_{n,m}^{-1}
=
x_{-\ep_j}(\frak{p}^m)
\quad\text{and}\quad
w_{n,m}\, x_{-\ep_j}(\frak{p}^m)\, w_{n,m}^{-1}
=
x_{\ep_j}(\frak{o}^m)
\]
by \eqref{E:z^mu conjugate}, it follows that $w_{n,m}$ normalizes $H_{n,m}$, and 
\begin{align*}
K_{n,m}
=w_{n,m}\, K_{n,m}\, w_{n,m}^{-1}
&=w_{n,m}\, \prod_{j=1}^n x_{\ep_j}(\frak{o})\,\prod_{j=1}^n x_{-\ep_j}(\frak{p}^m)\, H_{n,m}\, w_{n,m}^{-1}\\
&=\prod_{j=1}^n x_{-\ep_j}(\frak{p}^m)\,\prod_{j=1}^n x_{\ep_j}(\frak{o})\, H_{n,m}.
\end{align*}

We now turn to establishing the identity \eqref{E:K decomp}. It is clear from the second description of $K_{n,m}$ that the right-hand 
side of \eqref{E:K decomp} is contained in $K_{n,m}$. To prove the reverse containment, we show that for any $k\in K_{n,m}$, there 
exist $y_1,\ldots, y_n\in\frak{o}$ and $z_1,\ldots, z_n\in\frak{p}^m$ such that 
\begin{equation}\label{E:fix e_0}
\prod_{j=1}^n x_{-\ep_{n+1-j}}(-z_{n+1-j})\,\prod_{j=1}^n x_{\ep_{n+1-j}}(-y_{n+1-j})\, ke_0=e_0.
\end{equation}
Then since 
\[
H_{n,m}=\stt{h\in K_{n,m}\mid h e_0=e_0},
\]
the proof follows.

To verify \eqref{E:fix e_0}, let
\[
ke_0=\sum_{j=-n}^{n} a_j e_j
\quad
\text{for some $a_j\in\frak{o}$ for $-n\le j\le n$.}
\]
Note that $a_0\in 1+\frak{p}$ by the definition of $K_{n,m}$. Since $k^{-1}(\varpi^m e_j)$ and $k^{-1}e_{-j}$ are both contained in 
$\bbL_{n,m}$, it follows that
\[
\varpi^m a_{-j}=\varpi^m\langle ke_0\,,\, e_j\rangle
=
\langle e_0\,,\, k^{-1}(\varpi^m e_{j})\rangle
\in
 2\frak{p}^m
\]
for $1\le j\le n$. A similar argument shows
\[
a_{j}
=
\langle ke_0\,,\, e_{-j}\rangle
=
\langle e_0\,,\,k^{-1}e_{-j}\rangle
\in
2\frak{p}^{m}.
\]
Consequently, we can write $a_{-j}=2b_{-j}$ and $a_{j}=2\varpi^m b_{j}$ for some $b_{\pm j}\in\frak{o}$.

To proceed, let $y\in F$ be a variable. We have 
\begin{align*}
x_{\ep_1}(-y)ke_0
=
2(b_{-n}+a_0y-\varpi^mb_1y^2)e_{-n}
+
\sum_{j=-n+1}^{-1} 2b_j e_j
+
c(y)e_0
+
\sum_{j=1}^n 2\varpi^m b_je_j,
\end{align*}
where $c(y):=a_0-2\varpi^mb_1y$. Since $a_0\in 1+\frak{p}$, Hensel's lemma implies that there exists $y_1\in\frak{o}$ such that 
\[
b_{-n}+a_0y_1-\varpi^mb_1y_1^2=0.
\]
Note that $c_1:=c(y_1)=a_0-2\varpi^m b_1y_1\in 1+\frak{p}$. Continue this process, we can find $y_2,\ldots, y_n\in\frak{o}$ such that 
\[
\prod_{j=1}^n x_{\ep_{n+1-j}}(-y_{n+1-j})\,ke_0
=
c_ne_0+\sum_{j=1}^n 2\varpi^m b_j e_j,
\]
where $c_n:=a_0+2\varpi^m\sum_{j=1}^nb_jy_j\in 1+2\frak{p}$. 

Similarly, if $z_j=\varpi^m b_j\in\frak{p}^m$ for $1\le j\le n$, then a direct computation shows that
\[
\prod_{j=1}^n x_{-\ep_{n+1-j}}(-z_{n+1-j})\,\prod_{j=1}^n x_{\ep_{n+1-j}}(-y_{n+1-j})\,ke_0= c_ne_0.
\]
Set
\[
h=\prod_{j=1}^n x_{-\ep_{n+1-j}}(-z_{n+1-j})\,\prod_{j=1}^n x_{\ep_{n+1-j}}(-y_{n+1-j})\,k\in K_{n,m}.
\]
Then $he_0=c_ne_0$ and hence
\[
2c_n^2=\langle he_0\,,\,he_0\rangle=\langle e_0\,,\,e_0\rangle=2,
\]
which implies $c_n^2=1$. Since $c_n\in 1+2\frak{p}$, we have $c_n=1$, and therefore $h\in H_{n,m}$, as desired. 
\end{proof}

\begin{Remark}\label{R:n+1 column}\noindent
\begin{itemize}
\item[(1)]
From the proof of \lmref{L:K decomp}, we deduce that if $k=(k_{i,j})\in K_{n,m}$ with $m\ge 1$, then $k_{i,n+1}\in 2\,\frak{o}$ for
$1\le i\le n$, $k_{i,n+1}\in 2\,\frak{p}^m$ for $n+2\le i\le 2n$ and $k_{n+1,n+1}\in 1+2\,\frak{p}^m$.
\item[(2)]
Together with \eqref{E:H_r,e} and \eqref{E:H_r,m}, \lmref{L:K decomp} provides useful information about $K_{n,m}$.
\end{itemize}
\end{Remark}

\subsection{Another family of open compact subgroups}
Sometimes, it is more convenient to consider the open compact subgroups defined by 
\begin{equation}\label{E:K^0}
J^0_{n,m}=\varpi^{\lfloor\frac{m}{2}\rfloor\lambda_n}\, J_{n,m}\,\varpi^{-\lfloor\frac{m}{2}\rfloor\lambda_n}
\quad\text{and}\quad
K^0_{n,m}=\varpi^{\lfloor\frac{m}{2}\rfloor\lambda_n}\, K_{n,m}\,\varpi^{-\lfloor\frac{m}{2}\rfloor\lambda_n}.
\end{equation}
These open compact subgroups have the following properties:
\begin{equation}\label{E:property of K^0}
K^0_{n,i}=K_{n,i}\,(i=0,1),\quad H_{n,e}\subset K^0_{n,m},
 \quad\text{and}\quad
 w_{\ep_j,e}\in J^0_{n,m}\setminus K^0_{n,m}\,(m\ge 1), 
 \end{equation}
where $e=0,1$ satisfies $m\equiv e\pmod{2}$. Moreover, we have

\begin{cor}\label{C:K^0 decomp}
Let $m\ge 1$ and write $m=e+2\el$ for some $e=0,1$ and $\el\ge 0$. Then 
\begin{align*}
K^0_{n,m}
&=
\prod_{j=1}^n x_{\ep_j}\left(\frak{p}^\el\right)\,\prod_{j=1}^n x_{-\ep_j}\left(\frak{p}^{e+\el}\right)\, H_{n,e}\\
&=
\prod_{j=1}^n x_{-\ep_j}\left(\frak{p}^{e+\el}\right)\,\prod_{j=1}^n x_{\ep_j}\left(\frak{p}^\el\right)\, H_{n,e}.
\end{align*}
In particular, we have
\begin{equation}\label{E:descending}
\quad K^0_{n,m}\supset K^0_{n,m+2},
\end{equation}
for every $m\ge 0$.
\end{cor}

\begin{proof}
This follows from the definition of $K^0_{n,m}$, the identity \eqref{E:H_r,m} and \lmref{E:K decomp}.
\end{proof}

\section{Double coset decompositions}\label{S:DCD}
The aim of this section is to obtain explicit double coset representatives  $\b{\Sigma}_{r,m}$ for
\begin{equation}\label{E:double coset}
\b{P}_{\a_r}(F)\backslash\SO_{2n+1}(F)/K^0_{n,m}
\end{equation}
for $1\le r\le n$ and $m\ge 0$. Here, $\b{P}_{\a_r}(F)=M_{\a_r}(F)\b{N}_{\a_r}(F)$ denotes the opposite of $P_{\a_r}(F)$. 
We work with $K^0_{n,m}$ rather than $K_{n,m}$ because of the descending property \eqref{E:descending} of $K^0_{n,m}$.

We begin with a simple lemma. To this end, let $\b{B}_n=T_n\b{U}_n$ be the Borel subgroup of $\SO_{2n+1}$ that is opposite to 
$B_n$, where $\b{U}_n$ is the unipotent radical of $\b{B}_n$.

\begin{lm}\label{L:Iwasawa decomp}
Let $e=0,1$. Then we have the decomposition 
\[
\SO_{2n+1}(F)
=
B_n(F)\,J_{n,e}
=
\b{B}_n(F)\,J_{n,e}.
\]
\end{lm}

\begin{proof}
Since
\[
w_{n,e}\,B_n(F)\,w_{n,e}^{-1}=\b{B}_n(F), 
\]
it suffices to establish the first equality. 

If $e=0$, then $J_{n,0}$ is the hyperspecial maximal compact subgroup of $\SO_{2n+1}(F)$, and the assertion follows from the 
standard Iwasawa decomposition.

To handle the case $e=1$, note that the the images of $J_{n,0}\cap J_{n,1}$ and $J_{n,0}$, under the reduction modulo 
$\frak{p}$, are $P_{\a_n}(\frak{f})$ and $\SO_{2n+1}(\frak{f})$, respectively. By lifting the Bruhat decomposition of 
$\SO_{2n+1}(\frak{f})$ to $J_{n,0}$, we obtain the decomposition
\[
J_{n,0}=\bigcup_{w\in W_n^0} \left(U_n(F)\cap J_{n,0}\right)\,w\,\left(J_{n,0}\cap J_{n,1}\right).
\]
Here $W_n^0\simeq\stt{\pm 1}^n$ is the subgroup of $J_{n,0}$ generated by the elements $w_{\ep_\el}$ for $1\le\el\le n$. 

Since $W^0_n$ is contained in $T_n(F)\,J_{n,1}$, it follows that
\begin{align*}
\SO_{2n+1}(F)
=
B_n(F)\,J_{n,0}
=
\bigcup_{w\in W^0_n} B_n(F)\,w\,\left(J_{n,0}\cap J_{n,1}\right)
\subseteq 
B_n(F)\,J_{n,1}
\subseteq
\SO_{n+1}(F),
\end{align*}
which shows $\SO_{2n+1}(F)=B_n(F)\,J_{n,1}$.
\end{proof}

Now we can state and prove the main result of this section.

\begin{prop}\label{P:set of rep}
Let $e=0,1$ such that $m\equiv e\pmod 2$. Then sets $\b{\Sigma}_{r,m}$ of representatives of 
\[
\b{P}_{\a_r}(F)\backslash\SO_{2n+1}(F)/K^0_{n,m}
\]
can be chosen as follows\footnote{When $r=n$ and $m=e=0$, we understand that $\b{\Sigma}_{n,0}=\stt{x_{\ep_n}(1)}$.}:
\[
\b{\Sigma}_{r,m}
=
\begin{cases}
\stt{x_{\ep_r}(\varpi^i)\mid 0\le i\le \lfloor\frac{m}{2}\rfloor}\quad&\text{if $1\le r\le n-1$},\\
\stt{x_{\ep_n}(\varpi^i),\,x_{\ep_n}(\varpi^j)w_{\ep_n,e}\mid 
0\le i \le \lfloor\frac{m}{2}\rfloor,\,1-e\le j\le \lfloor\frac{m}{2}\rfloor}\quad&\text{if $r=n$}.
\end{cases}
\] 
\end{prop}

\begin{proof}
Since $\b{P}_{\a_r}(F)\supset\b{B}_n(F)$ for every $1\le r\le n$, we obtain from \lmref{L:Iwasawa decomp} that 
\[
\SO_{2n+1}(F)=\b{P}_{\a_r}(F)\,J_{n,e},
\]
for $1\le r\le n$, and $e=0,1$. In particular, the first assertion follows from this decomposition and the fact that 
$K^0_{n,0}=K_{n,0}=J_{n,0}$, i.e. the set $\b{\Sigma}_{r,0}$ may consist of an arbitrary element of $\SO_{2n+1}(F)$. We choose 
$\b{\Sigma}_{r,0}=\stt{x_{\ep_r}(1)}$ so that \propref{P:gp of int K^0} in \S\ref{S:gp from int} can be stated uniformly.

For $m\ge 1$, we actually prove that $\b{\Sigma}'_{r,m}$ forms a set of representatives of the double coset \eqref{E:double coset}, so 
as to unify the proof, where 
\[
\b{\Sigma}'_{r,m}
=
\begin{cases}
\stt{x_{\ep_1}(\varpi^i)\mid 0\le i\le \lfloor\frac{m}{2}\rfloor}\quad&\text{if $1\le r\le n-1$},\\
\stt{x_{\ep_1}(\varpi^i),\,x_{\ep_1}(\varpi^j)w_{\ep_1,e}\mid 
0\le i \le \lfloor\frac{m}{2}\rfloor,\,1-e\le j\le \lfloor\frac{m}{2}\rfloor}\quad&\text{if $r=n$}.
\end{cases}
\] 
Then since 
\[
\b{P}_{\a_r}(F)x_{\ep_1}\left(\varpi^i\right)K^0_{n,m}
=
\b{P}_{\a_r}(F)x_{\ep_r}\left(\varpi^i\right)K^0_{n,m},
\]
for $1\le r\le n$ and 
\[
\b{P}_{\a_n}(F)x_{\ep_1}\left(\varpi^j\right)w_{\ep_1,e}K^0_{n,m}
=
\b{P}_{\a_n}(F)x_{\ep_n}\left(\varpi^j\right)w_{\ep_n,e}K^0_{n,m},
\]
the proof follows.

To verify the case $m=1$, note that 
\begin{equation}\label{E:SO/K_n,1 decomp}
\SO_{2n+1}(F)
=
\b{P}_{\a_r}(F)\,J_{n,1}
=
\b{P}_{\a_r}\,K_{n,1}\cup\b{P}_{\a_r}\,w_{\ep_n,1}\,K_{n,1}
\quad\text{for $1\le r\le n$.}
\end{equation}
Since $w_{\ep_n,1}\in\b{P}_{\a_r}(F)$ for $r<n$, we obtain
$\SO_{2n+1}(F)=\b{P}_{\a_r}(F)\,K_{n,1}$, which establishes the case for $r<n$. On the other hand, if $r=n$, and if we have
$w_{\ep_n,1}=hk$ for some $h\in\b{P}_{\a_n}(F)$ and $k\in K_{n,1}$, then 
\[
h
=
w_{\ep_n,1}\,k^{-1}\in\b{P}_{\a_n}(F)\cap J_{n,1}\subset K_{n,1},
\]
which implies
\[
w_{\ep_n,1}\in K_{n,1},
\]
leading to a contradiction. Therefore, \eqref{E:SO/K_n,1 decomp} is a disjoint union when $r=n$. Since 
\[
\b{P}_{\a_n}(F)w_{\ep_n,1}\, K_{n,1}
=
\b{P}_{\a_n}(F)\,w_{\ep_1,1}\,K_{n,1},
\]
the assertion for $m=1$ is verified.

We proceed to justify the case $m\ge 2$. Write 
\[
m=e+2\el,
\]
where $\el=\lfloor\frac{m}{2}\rfloor\ge 1$. Since the proof for $m\ge 2$ is quite lengthy, we divide it into 
three steps for clarity.

\subsection*{Step 1}
We first claim that the set $\b{\Sigma}'_{r,m}$ may be chosen from one of the following subsets, depending on $r$:
\begin{itemize}
\item If $1\le r<n$, then
\begin{equation}\label{E:1st rep r<n}
\stt{\prod_{i=1}^rx_{\ep_i}(\varpi^{c_i})\mid 0\le c_1\le\cdots\le c_r\le\ell}.
\end{equation}
\item If $r=n$, then
\begin{equation}\label{E:1st rep r=n}
\stt{\prod_{i=1}^n x_{\ep_i}(\varpi^{c_i}),\quad\left(\prod_{i=1}^n x_{\ep_i}(\varpi^{d_i})\right)w_{\ep_1,e}
\mid 0\le c_1\le\cdots\le c_r\le\ell,\quad 0\le d_1\le\cdots\le d_r\le\ell}.
\end{equation}
\end{itemize}

Assume that $e=1$. Then the previous arguments show that
\[
\SO_{2n+1}(F)=\b{P}_{\a_r}(F)\,K_{n,1}
\quad\text{or}\quad
\SO_{2n+1}(F)=\b{P}_{\a_n}(F)K_{n,1}\bigsqcup\b{P}_{\a_n}(F)\,w_{\ep_1,1}\,K_{n,1},
\]
according to $1\le r<n$ or $r=n$, respectively. Since $w_{\ep_1,1}$ normalizes $K_{n,1}$, we have
\[
\b{P}_{\a_r}(F)\backslash\SO_{2n+1}(F)/K^0_{n,m}
=
\left(\b{P}_{\a_r}(F)\cap K_{n,1}\right)\backslash K_{n,1}/K^0_{n,m} 
\]
for $1\le r< n$, and
\begin{align*}
\b{P}_{\a_n}(F)\backslash\SO_{2n+1}(F)/K^0_{n,m}
=&
\left(\b{P}_{\a_n}(F)\cap K_{n,1}\right)\backslash K_{n,1}/K^0_{n,m}\\
&\bigsqcup
\left(\b{P}_{\a_n}(F)\cap K_{n,1}\right)\backslash w_{\ep_1,1}\,K_{n,1}/K^0_{n,m}
\end{align*}
for $r=n$. On the other hand,  \lmref{L:K decomp}, \eqref{E:H_r,m} and \eqref{E:K^0} imply
\[
K_{n,1}
=
\left(\b{P}_{\a_r}(F)\cap K_{n,1}\right)\,\prod_{j=1}^r x_{\ep_j}(\cP_\el)\,K^0_{n,m}
\]
for $1\le r\le n$, where $\cP_\el$ denotes a set of representatives of $\frak{o}/\frak{p}^\el$. 
Furthermore, since $w_{\ep_1,1}$ normalizes $\SO_{2n}(F)$, and hence normalizes $H_{n,1}$, the same results imply
\[
w_{\ep_1,1}\,K_{n,1}
=
K_{n,1}\,w_{\ep_1,1}
=
\left(\b{P}_{\a_n}(F)\cap K_{n,1}\right)\,\prod_{j=1}^n x_{\ep_j}(\cP_\el)w_{\ep_1,1}\,K^0_{n,m}
\]
for $r=n$. 

Since we are allowing conjugation of the elements in $T_n(\frak{o})$ and the Weyl elements $w_{\ep_i-\ep_j}$ for 
$1\le i<j\le r$ (when $r\ge 2$), the claim for the case $r<n$ and $e=1$ follows. For $r=n$, in addition to conjugating the aforementioned 
elements, we may also need to multiply on the right by the elements of the form 
\[
w_{\ep_j,1}w_{\ep_1,1}\in H_{n,1},
\] 
with $2\le j\le n$, in order to obtain the desired result. In this way, the first claim for $e=1$ follows.

Suppose that $e=0$. By \lmref{L:Iwasawa decomp} and the fact that $J_{n,0}=K_{n,0}$, we have
\[
\b{P}_{\a_r}(F)\backslash\SO_{2n+1}(F)/K^0_{n,m}
=
\left(\b{P}_{\a_r}(F)\cap K_{n,0}\right)\backslash K_{n,0}/K^0_{n,m}.
\]
In this case, we do not have a decomposition for $K_{n,0}$ similar to that for $K_{n,1}$. Instead, we apply the 
following decomposition:
\[
K_{n,0}
=
\b{U}_n(\frak{o})\,U_n(\frak{o})\,\left(N_{\SO_{2n+1}(F)}(T_n(F))\cap K_{n,0}\right),
\]
by \cite[Proposition 6.4.9 (iii)]{BruhatTits1972} (see also \cite[Proposition 7.3.12 (1)]{KalethaPrasad2023}). 
At this point, note that $\stt{I_{2n+1}, w_{\ep_i}}$ forms a set of representatives of 
\[
\left(N_{\SO_{2n+1}(F)}(T_n(F))\cap K_{n,0}\right)/\left(N_{\SO_{2n+1}(F)}(T_n(F))\cap H_{n,0}\right)
\]
for any $1\le i\le n$ by \eqref{E:Weyl gp quo}. Since $w_{\ep_i}$ normalizes $K^0_{n,m}$ by \eqref{E:K^0}, and we have
\[
U_n(\frak{o})K^0_{n,m}=\prod_{j=1}^n x_{\ep_j}(\frak{o})\,K^0_{n,m},
\]
it follows that
\begin{align*}
K_{n,0}
&=
\b{U}_n(\frak{o})\,U_n(\frak{o})\,K^0_{n,m}
\,\bigcup\,
\b{U}_n(\frak{o})\,U_n(\frak{o})\,K^0_{n,m}\,w_{\ep_i}\\
&=
\b{U}_n(\frak{o})\,\prod_{j=1}^n x_{\ep_j}(\frak{o})\,K_{n,m}^0
\,\bigcup\,
\b{U}_n(\frak{o})\,\prod_{j=1}^n x_{\ep_j}(\frak{o})\,w_{\ep_i}\,K_{n,m}^0.
\end{align*}

At this point, assume first that $1\le r<n$ and take $i=n$. Then since
\[
\b{U}_n(\frak{o})\,\prod_{j=r+1}^n x_{\ep_j}(\frak{o})
\subset \b{P}_{\a_r}(F),
\quad
w_{\ep_n}\in\b{P}_{\a_r}(F)
\quad\text{and}\quad
\prod_{j=1}^r x_{\ep_j}(\frak{o})\,w_{\ep_i}=w_{\ep_n}\,\prod_{j=1}^r x_{\ep_j}(\frak{o}),
\]
we conclude from \lmref{L:K decomp} (for $m\ge 2$) that
\[
K_{n,0}
=
\left(\b{P}_{\a_r}(F)\cap K^0_{n,m}\right)\,\prod_{j=1}^r x_{\ep_j}\left(\cP_\el\right)\,K^0_{n,m}.
\]
Now, we are in the situation similar to the case $e=1$. 

On the other hand, when $r=n$, we choose $i=1$. Then since 
\[
\b{U}_n(\frak{o})\subset \b{P}_{\a_n}(F)
\quad\text{and}\quad
w_{\ep_1}\,K^0_{n,m}=K^0_{n,m}\,w_{\ep_1},
\]
we can apply the same lemma to obtain
\begin{align*}
K_{n,0}
=&
\left(\b{P}_{\a_r}(F)\cap K^0_{n,m}\right)\,\prod_{j=1}^n x_{\ep_j}\left(\cP_\el\right)\,K^0_{n,m}\\
&\bigcup
\left(\b{P}_{\a_r}(F)\cap K^0_{n,m}\right)\,\prod_{j=1}^n x_{\ep_j}\left(\cP_\el\right)\,w_{\ep_1}\,K^0_{n,m}.
\end{align*}
Again, we are in a situation analogous to the case $e=1$; hence, the same argument applies here, establishing the first claim.

\subsection*{Step 2}
Our next claim is that $\b{\Sigma}'_{r,m}$ contains a set of representatives for the desired double coset of $\SO_{2n+1}(F)$.
Note that when $r=1$, the assertion has already been established in the first step, so we may assume that $2\le r\le n$. 
To proceed, we need the following identity borrowed from \cite[(5.16)]{RobertsSchmidt2007}:
\begin{equation}\label{E:key id}
x_{\ep_{k-1}}\left(\varpi^i\right)x_{\ep_k}\left(\varpi^j\right)
=
x_{-\ep_{k-1}+\ep_k}\left(\varpi^{i+j}\right)x_{\ep_{k-1}}\left(\varpi^i\right)
x_{-\ep_{k-1}+\ep_k}\left(-\varpi^{j-i}\right)x_{\ep_{k-1}+\ep_k}\left(\varpi^{i+j}\right),
\end{equation}
for $i,j\in\bbZ$ and $2\le k\le n$.

To verify the second claim, we first show that if 
\[
g=\prod_{i=1}^r x_{\ep_i}\left(\varpi^{d_i}\right)
\] 
is contained in the sets given by \eqref{E:1st rep r<n} or \eqref{E:1st rep r=n} according to $r<n$ or $r=n$, respectively,  then 
\[
\b{P}_{\a_r}(F)\,g\,K^0_{n,m}
=
\b{P}_{\a_r}(F)\,s\,K^0_{n,m}
\]
for some $s\in\b{\Sigma}'_{r,m}$. In fact, we will show that $s=x_{\ep_1}(\varpi^{d_1})$. 

To begin with, write 
\[
g=g_1 x_{\ep_{r-1}}\left(\varpi^{d_{r-1}}\right) x_{\ep_r}\left(\varpi^{d_r}\right),
\]
where $g_1=\prod_{i=1}^{r-2}x_{\ep_i}\left(\varpi^{d_i}\right)$ with $g_1:=I_{2n+1}$ when $r=2$. Now, since 
\[
g_1\,x_{-\ep_{r-1}+\ep_r}(y)
=
x_{-\ep_{r-1}+\ep_r}(y)\,g_1,
\quad
x_{-\ep_{r-1}+\ep_r}(y)\in\b{P}_{\a_r}(F),
\]
and
\[
x_{-\ep_{r-1}+\ep_r}\left(-\varpi^{d}\right)x_{\ep_{r-1}+\ep_r}\left(\varpi^{d'}\right)\in K^0_{n,m},
\]
for any $y\in F$ and non-negative integers $d, d'$, we can apply \eqref{E:key id} (with $k=r$, $i=d_{r-1}$ and $j=d_r$) to derive
\begin{align*}
\b{P}_{\a_r}(F)\,g\,K^0_{n,m}
=
\b{P}_{\a_r}(F)\,g_1\,x_{\ep_{r-1}}(\varpi^{d_{r-1}})\,x_{\ep_r}\left(\varpi^{d_r}\right)\,K^0_{n,m}
=
\b{P}_{\a_r}(F)\,g_1\,x_{\ep_{r-1}}(\varpi^{d_{r-1}})\,K^0_{n,m}.
\end{align*}
Continue this process, we finally obtain
\[
\b{P}_{\a_r}(F)\,g\,K^0_{n,m}
=
\b{P}_{\a_r}(F)\,x_{\ep_1}(\varpi^{d_1})\,K^0_{n,m},
\]
as desired. 

To finish the second step, it remains to verify that if $r=n$ and 
$g=\left(\prod_{j=1}^n x_{\ep_j}\left(\varpi^{d_j}\right)\right)w_{\ep_1,e}$ is contained 
in \eqref{E:1st rep r=n}, then 
\[
\b{P}_{\a_r}(F)\,g\,K^0_{n,m}
=
\b{P}_{\a_r}(F)\,s\,K^0_{n,m}
\]
for some $s\in\b{\Sigma}'_{n,m}$. Indeed, since $w_{\ep_1,e}$ normalizes $K_{n,m}^0$, the previous argument implies
\begin{align*}
\b{P}_{\a_n}(F)\,g\,K^0_{n,m}
&=
\b{P}_{\a_n}(F)\,\left(\prod_{j=1}^n x_{\ep_j}\left(\varpi^{d_j}\right)\right)\,K^0_{n,m}\,w_{\ep_1,e}\\
&=
\b{P}_{\a_n}(F)\, x_{\ep_1}\left(\varpi^{d_1}\right)\,K^0_{n,m}\,w_{\ep_1,e}\\
&=
\b{P}_{\a_n}(F)\, x_{\ep_1}\left(\varpi^{d_1}\right)w_{\ep_1,e}\,K^0_{n,m}.
\end{align*}
Finally, if $e=0$, then the identity
\[
x_{\ep_1}(1)=x_{-\ep_1}(1)x_{\ep_1}(-1)w_{\ep_1},
\]
and the fact that $x_{-\ep_1}(1)\in\b{P}_{\a_n}(F)$ imply 
\begin{equation}\label{E:case r=n, e=0}
\b{P}_{\a_n}(F)\,x_{\ep_1}(1)\,K^0_{n,m}
=
\b{P}_{\a_n}(F)\,x_{\ep_1}(-1)w_{\ep_1}\,K^0_{n,m}
=
\b{P}_{\a_n}(F)\,x_{\ep_1}(1)w_{\ep_1}\,K^0_{n,m}.
\end{equation}
This establishes the second step.

\subsection*{Step 3}
The last step is to show that elements in $\b{\Sigma}'_{r,m}$ represent distinct double cosets.
Suppose in contrary that
\[
x_{\ep_1}(\varpi^c)=p\,x_{\ep_1}(\varpi^d)\,k
\]
for some integers $0\le d< c\le \el$, $p\in \b{P}_{\a_r}(F)$ and $k\in K^0_{n,m}$. Then, since
\[
p=x_{\ep_1}(\varpi^c)\,k^{-1}\,x_{\ep_1}(-\varpi^d)\in\b{P}_{\a_r}(F)\cap K^0_{n,m},
\]
we have
\[
p^{-1}e_i
=
\sum_{j=n-r+1}^n p_{i,j}e_j
\]
for some $p_{i,j}\in\frak{o}$ with $n-r+1\le i,j\le n$. On the other hand, write
\[
k e_0
=
\sum_{t=-n}^{n} a_t e_{t}. 
\]
By Remark \ref{R:n+1 column} and \eqref{E:K^0}, we have 
\[
a_{-t}\in 2\frak{p}^\el,\quad a_0\in 1+2\frak{p}^m\quad\text{and}\quad a_t\in 2\frak{p}^{e+\el}
\quad\text{for $1\le t\le n$}. 
\]
Now, a direct computation gives
\[
x_{\ep_1}(\varpi^d)ke_0 
=
\left(a_{-n}-2\varpi^da_0-\varpi^{2d}a_n\right)e_{-n}
+
(a_0+\varpi^d a_n)e_0
+
\sum_{-n+1\le t\ne 0\le n}a_t e_t.
\]
It is important to observe that 
\[
a_{-n}-2\varpi^da_0-\varpi^{2d}a_n\in 2\varpi^d\frak{o}^\x.
\]

We proceed to compute 
\[
\langle x_{\ep_1}(\varpi^c)e_0, e_i\rangle
\]
for $n-r+1\le i\le n$ in two different ways. First, we have
\[
\langle x_{\ep_1}(\varpi^c)e_0, e_i\rangle
=
\langle e_0-2\varpi^c e_{-n}, e_i\rangle
=
-2\varpi^c\delta_{i,n}.
\]
Next, we also have
\[
\langle x_{\ep_1}(\varpi^c)e_0, e_i\rangle
=
\langle x_{\ep_1}(\varpi^d)ke_0, p^{-1}e_i\rangle
=
\left(a_{-n}-2\varpi^da_0-\varpi^{2d}a_n\right)p_{i,n}
+
\sum_{j=n-r+1}^{n-1}a_jp_{i,j}.
\]
It follows that 
\begin{equation}\label{E:key id 2}
\left(a_{-n}-2\varpi^da_0-\varpi^{2d}a_n\right)p_{i,n}
+
\sum_{j=n-r+1}^{n-1}a_jp_{i,j}
=
-2\varpi^c\delta_{i,n},
\end{equation}
for $n-r+1\le i\le n$.

By applying \eqref{E:key id 2} for $n-r+1\le i\le n-1$, and noting that $d<\el$, we deduce that $p_{i,n}\in\frak{p}$ for $n-r+1\le i\le n-1$.
This implies $p_{n,n}\in\frak{o}^\x$ since $p^{-1}\in\b{P}_{\a_r}(F)\cap K^0_{n,m}$. Applying the same identity for $i=n$, we obtain
\[
a_{-n}-2\varpi^da_0-\varpi^{2d}a_n
\in
2\frak{p}^c,
\]
which contradicts to the fact that $a_{-n}-2\varpi^da_0-\varpi^{2d}a_n\in 2\varpi^d\frak{o}^\x$ and $d<c$. We thus complete the proof 
for the case $1\le r<n$.

Suppose from now on that $r=n$. If 
\[
\b{P}_{\a_n}(F)\,x_{\ep_1}\left(\varpi^c\right)w_{\ep_1,e}\,K^0_{n,m}
=
\b{P}_{\a_n}(F)\,x_{\ep_1}\left(\varpi^d\right)w_{\ep_1,e}\,K^0_{n,m}
\]
for some $1-e\le c,d\le \el$, then since $w_{\ep_1,e}$ normalizes $K^0_{n,m}$, we have
\[
\b{P}_{\a_n}(F)\,x_{\ep_1}\left(\varpi^c\right)\,K^0_{n,m}\,w_{\ep_1,e}
=
\b{P}_{\a_n}(F)\,x_{\ep_1}\left(\varpi^d\right)\,K^0_{n,m}\,w_{\ep_1,e},
\]
which implies $c=d$.

We still need to show that 
\[
\b{P}_{\a_n}(F)\,x_{\ep_1}\left(\varpi^i\right)\,K^0_{n,m}
\neq 
\b{P}_{\a_n}(F)\,x_{\ep_1}\left(\varpi^j\right)w_{\ep_1,e}\,K^0_{n,m}
\]
for $0\le i\le \el$ and $1-e\le j\le \el$, except for $e=i=j=0$. When $e=1$, this follows from the facts that 
\[
\b{P}_{\a_n}(F)\,x_{\ep_1}\left(\varpi^i\right)\,K^0_{n,m}\subset\b{P}_{\a_n}(F)K_{n,1},
\]
and
\[
\b{P}_{\a_n}(F)\,x_{\ep_1}\left(\varpi^j\right)w_{\ep_1,e}\,K^0_{n,m}\subset\b{P}_{\a_n}(F)\,w_{\ep_1,1}\,K_{n,1},
\]
together with the disjointness of $\b{P}_{\a_n}(F)K_{n,1}$ and $\b{P}_{\a_n}(F)\,w_{\ep_1,1}\,K_{n,1}$.

Suppose that $e=0$ and that 
\[
\b{P}_{\a_n}(F)\,x_{\ep_1}\left(\varpi^i\right)\,K^0_{n,m}
=
\b{P}_{\a_n}(F)\,x_{\ep_1}\left(\varpi^j\right)w_{\ep_1}\,K^0_{n,m}
\]
for $0\le i\le \el$ and $1\le j\le \el$. Since $w_{\ep_1}$ normalizes $K^0_{n,m}$, we may assume without loss of generality that 
$i\le j$.  If $i=0$, then the identity \eqref{E:case r=n, e=0} implies 
\[
\b{P}_{\a_n}(F)\,x_{\ep_1}\left(1\right)w_{\ep_1}\,K^0_{n,m}
=
\b{P}_{\a_n}(F)\,x_{\ep_1}\left(\varpi^j\right)w_{\ep_1}\,K^0_{n,m},
\]
which in turn implies that $j=0$. Therefore, we have $1\le i\le j$, and 
\[
x_{\ep_1}\left(\varpi^i\right)
=
p\,x_{\ep_1}\left(\varpi^j\right)w_{\ep_1}\,k
\]
for some $p\in\b{P}_{\a_n}(F)$ and $k\in K^0_{n,m}$. Since $i,j\ge 1$, both $x_{\ep_1}\left(\varpi^i\right)$ and 
$x_{\ep_1}\left(\varpi^j\right)$ belong to $K^0_{n,2}\subset J^0_{n,2}$ (see \eqref{E:K^0}). 
On the other hand, since $w_{\ep_1,2}\in J_{n,2}$, it follows that
$w_{\ep_1}\in J^0_{n,2}$; hence
\[
p\in\b{P}_{\a_n}(F)\cap J^0_{n,2}\subset K^0_{n,2}.
\]
But this implies that $w_{\ep_1}\in K^0_{n,2}$, which is not true. This completes the proof of the case $m\ge 2$ and, consequently,
the proof of the proposition.
\end{proof}

As an immediate consequence, we obtain

\begin{cor}\label{C:double coset}
Let $e=0,1$ such that $m\equiv e\pmod 2$. Then sets $\Sigma_{r,m}$ of representatives of 
\[
P_{\a_r}(F)\backslash\SO_{2n+1}(F)/K^0_{n,m}
\]
can be chosen as follows\footnote{When $r=n$ and $m=e=0$, we understand that $\Sigma_{n,0}=\stt{x_{-\ep_n}(1)}$.}:
\[
\Sigma_{r,m}
=
\begin{cases}
\stt{x_{-\ep_r}(\varpi^{i})\mid  e\le i\le\lceil\frac{m}{2}\rceil}\quad&\text{if $1\le r\le n-1$},\\
\stt{x_{-\ep_n}(\varpi^{i}),\,x_{-\ep_n}(\varpi^{j})w_{\ep_n,e}\mid 
e\le i\le\lceil\frac{m}{2}\rceil,\,1\le j\le\lceil\frac{m}{2}\rceil}\quad&\text{if $r=n$}.
\end{cases}
\] 
\end{cor}

\begin{proof}
Since
\[
P_{\a_r}(F)=w_{r,e}\,\b{P}_{\a_r}(F)\,w^{-1}_{r,e}, 
\]
and $K^0_{n,m}$ is normalized by $w_{r,e}$, the desired sets of representatives can be obtained by conjugating the sets 
$\b{\Sigma}_{r,m}$ from \propref{P:set of rep} with the element $w_{r,e}$. 

Let $1\le r\le n$ and $0\le i\le \lfloor\frac{m}{2}\rfloor$ be integers. A direct computation shows 
\[
w_{r,e} x_{\ep_r}\left(\varpi^i\right) w_{r,e}^{-1}
=
x_{-\ep_r}\left((-1)^r\varpi^{e+i}\right)
\quad
\text{and}
\quad
w_{n,e} w_{\ep_n, e} w_{n,e}^{-1}=w_{\ep_n,e}.
\]
On the other hand, since
\[
P_{\a_r}(F)\,x_{-\ep_r}\left((-1)^r\varpi^{e+i}\right)\,K^0_{n,m}
=
P_{\a_r}(F)\,x_{-\ep_r}\left(\varpi^{e+i}\right)\,K^0_{n,m}
\]
and
\[
P_{\a_n}(F)\,x_{-\ep_n}\left((-1)^n\varpi^{e+i}\right)w_{\ep_n,e}\,K^0_{n,m}
=
P_{\a_n}(F)\,x_{-\ep_n}\left(\varpi^{e+i}\right)w_{\ep_n,e}\,K^0_{n,m},
\]
the corollary follows.
\end{proof}

\section{Groups from intersections}\label{S:gp from int}
We retain the notation from the beginning of the previous section. Let $s\in\SO_{2n+1}(F)$ and $r,m$ be integers with $1\le r\le n$ and 
$m\ge 0$. Define a subgroup $\b{M}^s_{r,m}$ of $M_{\a_r}(F)$ (see \S\ref{SS:para}) by
\begin{align*}
\b{M}^s_{r,m}
&=
\stt{g\in M_{\a_r}(F)\mid\text{$\exists u\in\b{N}_{\a_r}(F)$ such that $s^{-1}gus\in K^0_{n,m}$}}\\
&=
\stt{g\in M_{\a_r}(F)\mid\text{$\exists u\in\b{N}_{\a_r}(F)$ such that $s^{-1}ugs\in K^0_{n,m}$}}.
\end{align*}
The following observations follow immediately from the definition:
\begin{itemize}
\item[(i)]
If $s\in K^0_{n,m}$, then  $\b{M}^s_{r,m}=M_{\a_r}(F)\cap K^0_{n,m}$.
\item[(ii)]
If $w\in\SO_{2n+1}(F)$ normalizes $K^0_{n,m}$, then $\b{M}^{sw}_{r,m}=\b{M}^s_{r,m}$.
\end{itemize}

The aim of this section is to determine $\b{M}^s_{r,m}$ explicitly for $s$ in the sets $\b{\Sigma}^0_{r,m}$ obtained in 
\propref{P:set of rep}. We note that it suffices to consider $\b{M}^s_{r,m}$ for $s=x_{\ep_r}\left(\varpi^i\right)$ with 
$0\le i\le \lfloor\frac{m}{2}\rfloor$. This deduction follows from the observation (ii), as well as from the fact that 
$w_{\ep_r,e}$ normalizes $K^0_{n,m}$. Here $e=0,1$ satisfies $m\equiv e\pmod{2}$, as usual. For simplicity, we also denote 
\[
s_{r,i}=x_{\ep_r}\left(\varpi^i\right)
\] 
for $i,j\in\bbZ$ with $1\le r\le n$. 

To describe the result, let $\Gamma^{'}_{r,m}\subset\GL_r(\frak{o})$ be the open compact subgroup defined by 
\[
\Gamma^{'}_{r,m}
=
\bordermatrix{
              & (r-1)     & 1       \cr
    (r-1) & \mathfrak{o} &\mathfrak{p}^m   \cr
    1     &\mathfrak{o}& \mathfrak{u}^m \cr
            }\cap {\rm GL}_r(\frak{o}).
\]

\begin{prop}\label{P:gp of int K^0}
Let $e=0,1$ such that $m\equiv e\pmod 2$. Then we have
\[
\b{M}^{s_{r,i}}_{r,m}
=
\stt{\diag{a,h,a^*}\in M_{\a_r}(F)\mid a\in\Gamma^{'}_{r,\lfloor\frac{m}{2}\rfloor-i},\, h\in K^0_{n-r,e+2i}}
\simeq
\Gamma^{'}_{r,\lfloor\frac{m}{2}\rfloor-i}\,\x\, K^0_{n-r, e+2i}
\]
for $0\le i\le\lfloor\frac{m}{2}\rfloor$.
\end{prop}

\begin{proof}
Since $x_{\ep_r}(1)\in K^0_{n,m}=K_{n,m}$ for $m=0,1$, the assertion for $m=0,1$ follows immediately from
the observation (i). Thus, we assume $m\ge 2$ for the remainder of the proof. In particular, this implies 
$\el:=\lfloor\frac{m}{2}\rfloor\ge 1$. We divide the proof into two cases: $r=1$ and $r>1$.

\subsection*{Case $r=1$}
The goal is to establish the equality:
\begin{equation}\label{E:gp of int for r=1}
\b{M}^{s_{1,i}}_{1,m}
=
\stt{\diag{a,h,a^*}\in M_{\a_1}(F)\mid a\in\frak{u}^{\el-i},\,h\in K^0_{n-1, e+2i}}
\end{equation}
for $0\le i\le \el$.

We show that the right-hand side (RHS) of \eqref{E:gp of int for r=1} is contained in the left-hand side (LHS). 
Assume first that
\[
g=\diag{a,h,a^{-1}}\in M_{\a_1}(F)
\]
with $a\in\frak{u}^{\el-i}$ and $h\in H_{n-1,e}$. In this case, we check directly that 
\[
s_{1,i}^{-1}\,g\,s_{1,i}\in K^0_{n,m}.
\]
To this end, write
\[
s_{1,i}
=
\begin{pmatrix}
1&x_i&y_i\\
&I_{2n-1}&x'_i\\
&&1
\end{pmatrix},
\]
where
\[
x_i
=
-2\varpi^i E^{1,2n-1}_{1, n},
\quad
x_i'
=
\varpi^iE^{2n-1,1}_{n,1}
\quad\text{and}\quad
y=-\varpi^{2i}.
\]
Then $s^{-1}_{1,i}$ can be written as
\[
s^{-1}_{1,i}
=
\begin{pmatrix}
1&-x_i&y_i\\
&I_{2n-1}&-x'_i\\
&&1
\end{pmatrix}.
\]
Since $h\in H_{n-1,e}$, we have 
\[
x_ih=x_i
\quad\text{and}\quad
hx'_i=x'_i.
\]
A direct computation shows
\[
s^{-1}_{1,i}\,g\,s_{1,i}
=
\begin{pmatrix}
a&(a-1)x_i& ay_i+y_ia^{-1}+2\varpi^{2i}\\
&h&x'_i(1-a^{-1})\\
&&a^{-1}
\end{pmatrix}
\in
K^0_{n,m}.
\]
This implies that $g\in\b{M}^{s_{i,i}}_{1,m}$.

In view of the decomposition \lmref{L:K decomp} (for $e+2i\ge 1$), the fact that
\[
K_{n,0}
=
\langle H_{n,0},\, x_{\pm\ep_j}(\frak{o})\mid 1\le j\le n\rangle\quad\text{(for $e=i=0$)},
\]
and the relation \eqref{E:K^0}, it remains to prove the claim for 
\[
g=x_{\ep_j}(y)
\]
with $y\in\frak{p}^i$, and 
\[
g=x_{-\ep_j}(z)
\]
with $z\in\frak{p}^{e+i}$, for $2\le j\le n$.

Assume that $g=x_{\ep_j}(y)$. In this case, let $u=x_{-\ep_1+\ep_j}\left(-\varpi^{-i}y\right)\in\b{N}_{\a_1}(F)\cap K^0_{n,e}$. We 
will show that 
\[
s^{-1}_{1,i}\,u\,g\,s_{1,i}\in H_{n,e}.
\]
Since $H_{n,e}\subset K^0_{n,m}$, it then follows that $g\in\b{M}^{s_{1,i}}_{1,m}$. To verify the inclusion, note that $u, g$ and 
$s_{1,i}$ all belong to $K^0_{n,e}$, so 
\[
s^{-1}_{1,i}\,u\,g\,s_{1,i}\in K^0_{n,e}.
\]
Thus, it suffices to show that $s^{-1}_{1,i}\,u\,g\,s_{1,i}\in\SO_{2n}(F)$, i.e. $s^{-1}_{1,i}\,u\,g\,s_{1,i}e_0=e_0$. Now, we compute
\begin{align*}
s^{-1}_{1,i}\,u\,g\,s_{1,i}e_0
&=
x_{\ep_1}\left(-\varpi^i\right)\,x_{-\ep_1+\ep_j}\left(-\varpi^{-i}y\right)\,x_{\ep_j}\left(y\right)\,x_{\ep_1}\left(\varpi^i\right)e_0\\
&=
x_{\ep_1}\left(-\varpi^i\right)\,x_{-\ep_1+\ep_j}\left(-\varpi^{-i}y\right)\,x_{\ep_j}\left(y\right)\left(e_0-2\varpi^i e_{-n}\right)\\
&=
x_{\ep_1}\left(-\varpi^i\right)\,x_{-\ep_1+\ep_j}\left(-\varpi^{-i}y\right)\left(e_0-2y e_{-n+j-1}-2\varpi^i e_{-n}\right)\\
&=
x_{\ep_1}\left(-\varpi^i\right)\left[e_0-2y e_{-n+j-1}-2\varpi^i\left(e_{-n}-\varpi^{-i}ye_{-n+j-1}\right)\right]\\
&=
x_{\ep_1}\left(-\varpi^i\right)\left(e_0-2\varpi^i e_{-n}\right)\\
&=
e_0,
\end{align*}
as wanted. 

If $g=x_{-\ep_j}(z)$, choose $u=x_{-\ep_1-\ep_j}(\varpi^{-i}z)$. Then similar argument and computation show 
\[
s^{-1}_{1,i}\,x_{-\ep_1-\ep_j}\left(\varpi^{-i} z\right)\,x_{-\ep_j}(z)\,s_{1,i}\in H_{n,e}\subset K^0_{n,m}.
\]
Therefore, $g\in\b{M}^{s_{1,i}}_{1,m}$, and the first inclusion is established.

We proceed to show that the LHS of \eqref{E:gp of int for r=1} is contained in the RHS. Since $s_{1,i}\in K^0_{n, e+2i}$ and 
$K^0_{n,m}\subseteq K^0_{n,e+2i}$, it follows that
\[
s_{1,i}K^0_{n,m}s_{1,i}^{-1}\cap\bar{P}_{\a_1}(F)\subseteq K^0_{n,e+2i}\cap\bar{P}_{\a_1}(F).
\]
This implies
\[
\bar{M}^{s_{1,i}}_{1,m}
\subset
\stt{\diag{a,h,a^{-1}}\in M_{\a_1}(F)\mid a\in\frak{o}^\x,\, h\in K^0_{n-1, e+2i}}.
\]
It remains to show that 
\[
g=\diag{a,h,a^{-1}}\in M_{\a_1}(F)
\quad\Longrightarrow\quad 
a\in\frak{u}^{\el-i}.
\]

Suppose that $s_{1,i}^{-1}\,gu\,s_{1,i}\in K^0_{n,m}$ for some $u\in\b{N}_{\a_1}(F)$, and denote $h=gu$. We compute
\begin{align*}
\langle s_{1,i}^{-1}\,h\,s_{1,i}e_0\,,\, e_n\rangle
&=
\langle h\,s_{1,i}e_0\,,\, s_{1,i}e_n\rangle\\
&=
\langle
h\,s_{1,i}e_0\,,\, e_{n}+\varpi^ie_0-\varpi^{2i}e_{-n}
\rangle\\
&=
\langle
s_{1,i}e_0\,,\, h^{-1}e_{n} 
\rangle
+
\langle
h\,s_{1,i}e_0\,,\, \varpi^ie_0-\varpi^{2i}e_{-n}
\rangle\\
&=
\langle
e_0-2\varpi^i e_{-n}\,,\, ae_n
\rangle
+
\varpi^i\langle
h\,s_{1,i}e_0\,,\,e_0-2\varpi^i e_{-n}
\rangle
+\varpi^{2i}\langle
h\,s_{1,i}e_0\,,\,e_{-n}
\rangle\\
&=
-2\varpi^i a
+
\varpi^i\langle
s^{-1}_{1,i}\,h\,s_{1,i}e_0\,,\,e_0
\rangle
+\varpi^{2i}\langle
s^{-1}_{1,i}\,h\,s_{1,i}e_0\,,\,e_{-n}
\rangle.
\end{align*}
In above computations, we used the facts that $h^{-1}e_n=a$ and $s_{1,i}e_{-n}=e_{-n}$.
Since $s_{1,i}^{-1}\,h\, s_{1,i}\in K^0_{n,m}$, it follows from Remark \eqref{R:n+1 column} and \eqref{E:K^0} that
\[
\langle s_{1,i}^{-1}\,h\,s_{1,i}e_0\,,\, e_n\rangle\in 2\,\frak{p}^\ell,
\quad
\langle s^{-1}_{1,i}\,h\,s_{1,i}e_0\,,\, e_0\rangle\in2\,\left(1+2\,\frak{p}^m\right)
\quad\text{and}\quad
\langle
s^{-1}_{1,i}\,h\,s_{1,i}e_0\,,\, e_{-n}
\rangle\in 2\,\frak{p}_E^{e+\ell}.
\]
These relations imply that $a\in\frak{u}^{\ell-i}$, thereby completing the proof of \eqref{E:gp of int for r=1}.

\subsection*{Case $r>1$}
We proceed to prove the case $r>1$. The argument is analogous to, and indeed builds upon, the case $r=1$. We first claim that 
\[
\stt{\diag{a,h,a^*}\in M_{\a_r}(F)\mid a\in\Gamma'_{r,\el-i},\, h\in K^0_{n-r, e+2i}}
\subset
\b{M}^{s_{r,i}}_{r,m}.
\]
The claim follows if we can show that both 
\[
\diag{I_r, h, I_r}
\quad\text{and}\quad
\diag{a, I_{2(n-r)+1}, a^*}
\] 
belong to $\b{M}^{s_{r,i}}_{r,m}$, where $h\in K^0_{n-r, e+2i}$ and $a\in\Gamma'_{r,\el-i}$.

To prove that
\[
\diag{I_r, h, I_r}\in\b{M}^{s_{r,i}}_{r,m}
\]
for $h\in K^0_{n-r, e+2i}$, let $n'=n-r+1<n$ and consider the embedding 
\[
\jmath: \SO_{2n'+1}(F)\hookto\SO_{2n+1}(F);\quad g'\mapsto\diag{I_{r-1}, g', I_{r-1}}.
\]
For clarity, we denote the roots of $\SO_{2n'+1}(F)$ by appending a $\prime$. Thus, $\ep'_1,\ldots, \ep'_{n'}$ are the standard 
basis for $X^*(T_{n'})$, and $\a'_1,\ldots, \a'_{n'}$ are the simple roots of the root system of $\SO_{2n'+1}$. We also denote 
$s'_{r,i}=x_{\ep'_r}\left(\varpi^i\right)$ for $r,i\in\bbZ$ with $1\le r\le n'$. At this point, observe the followings
\[
\jmath\left(s'_{1,i}\right)
=
s_{r,i},
\quad
\jmath\left(\b{N}_{\a'_1}(F)\right)
\subset
\b{N}_{\a_r}(F)
\quad\text{and}\quad
\jmath\left( K^0_{n', m}\right)
\subset
K^0_{n,m},
\]
for every $m\ge 0$.

Now, let $h\in K^0_{n-r, e+2i}$ and put $g=\diag{1,h,1}\in M_{\a'_1}(F)$. By the result for the case $r=1$, there exists 
$u\in \b{N}_{\a'_1}(F)$ such that $s'^{-1}_{1,i}\,gu\,s'_{1,i}\in K^0_{n',m}$. The aforementioned observations then imply
\[
s_{r,i}^{-1}\,\jmath(g)\jmath(u)\,s_{r,i}=\jmath\left(s'^{-1}_{1,i}\,gu\,s'_{1,i}\right)\in\jmath\left(K^0_{n',m}\right)\subset K^0_{n,m}.
\]
Since 
\[
\jmath(g)=\diag{I_r,h,I_r}
\quad\text{and}\quad
\jmath(u)\in\jmath\left(\b{N}_{\a'_1}(F)\right)\subset \b{N}_{\a_r}(F),
\]
we conclude that $\diag{I_r, h,I_r}\in\b{M}^{s_{r,i}}_{r,m}$, as desired.

Next, we verify that 
\[
\diag{a,I_{2(n-r)+1}, a^*}\in\b{M}^{s_{r,i}}_{r,m}
\quad\text{for $a\in\Gamma^{'}_{r,\el-i}$}.
\]
To do so, write
\[
s_{r,i}
=
\begin{pmatrix}
I_r&X_i&Y_i\\
&I_{2(n-r)+1}&X'_i\\
&&I_r
\end{pmatrix},
\]
where
\[
X_i=-2\varpi^i E^{r, 2(n-r)+1}_{r, n-r+1},
\quad 
X^{'}_i=\varpi^i E^{2(n-r)+1, r}_{n-r+1, 1}
\quad\text{and}\quad
Y_i=-\varpi^{2i}E^r_{r,1}.
\]
Then 
\[
s_{r,i}^{-1}
=
\begin{pmatrix}
I_r&-X_i&Y_i\\
&I_{2(n-r)+1}&-X^{'}_i\\
&&I_r
\end{pmatrix},
\]
and a simple computation shows
\[
s^{-1}_{r,i}\begin{pmatrix}a\\&I_{2(n-r)+1}\\&&a^*\end{pmatrix}s_{r,i}
=
\begin{pmatrix}
a&(a-I_r)X_i&aY_i+Y_ia^*-X_iX^{'}_i\\
&I_{2(n-r)+1}&X^{'}_i(I_r-a^*)\\
&&a^*
\end{pmatrix}\in K^0_{n,m},
\]
for $a\in\Gamma^{'}_{r,\el-i}$. This establishes the first inclusion.

It remains to prove the reverse inclusion, i.e.
\[
\b{M}^{s_{r,i}}_{r,m}
\subset
\stt{\diag{a,h,a^*}\in M_{\a_r}(F)\mid a\in\Gamma^{'}_{r,\el-i},\, h\in K^0_{n-r, e+2i}}.
\]
Let $g=\diag{a,h,a^*}\in\b{M}_{r,m}^{s_{r,i}}$ and $u\in\b{N}$ such that 
\[
s^{-1}_{r,i}\,gu\,s_{r,i}\in K^0_{n,m}.
\]
Since both $s_{r,i}$ and $K^0_{n,m}$ are contained in $K^0_{n,e+2i}$, it follows that $gu\in K^0_{n,e+2i}$. This, in turn,  implies that
\[
a\in\GL_r(\frak{o})=\Gamma^{'}_{r,0}
\quad\text{and}\quad
h\in K^0_{n-r, e+2i}.
\]
In particular, this establishes the reverse inclusion for the case $i=\el$. Thus, in the remainder of the proof, we may assume that $i<\el$.
In this case, it remains to show that $a\in\Gamma^{'}_{r,\el-i}$.

Let $g_1=gu$, $n-r+1\le j\le n$ and write
\[
g_1^{-1}e_j
=
\sum_{t=n-r+1}^n a_{j,t} e_t
\]
for some $a_{j,t}\in\frak{o}$. We must show that 
\[
a_{n,n}\in\frak{u}^{\el-i}
\quad\text{and}\quad
a_{j,n}\in\frak{p}^{\el-i}
\]
for $n-r+1\le j\le n-1$. Then the proof follows. To this end, consider 
\[
\langle s_{r,i}^{-1}\,g_1\,s_{r,i}e_0\,,\, e_j\rangle,
\]
which belongs to $2\,\frak{p}^\el$ by Remark \ref{R:n+1 column} and \corref{C:K^0 decomp}.

For $n-r+1\le j\le n-1$, we have $s_{r,i}e_j=e_j$. It follows that
\begin{align*}
\langle s_{r,i}^{-1}\,g_1\,s_{r,i}e_0\,,\, e_j\rangle
=
\langle s_{r,i}e_0\,,\, g^{-1}_1e_j\rangle
=
\langle e_0-2\varpi^i e_{-n}\,,\, g^{-1}_1e_j\rangle
=
-2\varpi^i a_{j,n}.
\end{align*}
This implies that $a_{j,n}\in\frak{p}^{\el-i}$ for $n-r+1\le j\le n-1$. On the other hand, for $j=n$, we have
\begin{align*}
\langle s_{r,i}^{-1}\,g_1\,s_{r,i}e_0\,,\, e_n\rangle
&=
\langle g_1\,s_{r,i}e_0\,,\,e_{n-r+1}+\varpi^i e_0-\varpi^{2i}e_{-n+r-1}\rangle\\
&=
\langle s_{r,i}e_0\,,\,g_1^{-1}e_{n-r+1}\rangle
+
\varpi^i\langle g_1\,s_{r,i}e_0\,,\,v_0-2\varpi^ie_{-n+r-1}\rangle
+
\varpi^{2i}\langle g_1\,s_{r,i}e_0\,,\,e_{-n+r-1}\rangle\\
&=
\langle e_0-2\varpi^{2i}e_{-n+r-1}\,,\,g^{-1}e_{n-r+1}\rangle
+
\varpi^i\langle g_1\,s_{r,i}e_0, s_{r,i}e_0\rangle
+
\varpi^{2i}\langle g_1\,s_{r,i}e_0, s_{r,i}e_{-n+r-1}\rangle\\
&=
-2\varpi^i a_{n,n}
+
\varpi^i\langle s^{-1}_{r,i}\,g_1\,s_{r,i}e_0\,,\,e_0\rangle
+
\varpi^{2i}\langle s^{-1}_{r,i}\,g_1\,s_{r,i}e_0\,,\,e_{-n+r-1}\rangle.
\end{align*}
Since $s^{-1}_{r,i}\,g_1\,s_{r,i}\in K^0_{n,m}$, Remark \ref{R:n+1 column} and \corref{C:K^0 decomp} imply
\[
\langle s^{-1}_{r,i}\,g_1\,s_{r,i}e_0\,,\,e_0\rangle\in 2\left(1+2\,\frak{p}^m\right)
\quad\text{and}\quad
\langle s^{-1}_{r,i}\,g_1\,s_{r,i}e_0\,,\,e_{-n+r-1}\rangle\in 2\,\frak{p}^{e+\el}.
\]
From these, we conclude that $a_{n,n}\in\frak{u}^{\el-i}$, which completes the proof.
\end{proof}

We now derive a corollary of \propref{P:gp of int K^0}. Let $s\in\SO_{2n+1}(F)$, $1\le r\le n$ and $m\ge 0$. 
Define a subgroup $M^s_{r,m}$ of $M_{\a_r}(F)$ similar to $\b{M}^s_{r,m}$ by
\begin{align*}
M^s_{r,m}
&=
\stt{g\in M_{\a_r}(F)\mid\text{$s^{-1}gus\in K^0_{n,m}$ for some $u\in N_{\a_r}(F)$}}\\
&=
\stt{g\in M_{\a_r}(F)\mid\text{$s^{-1}ugs\in K^0_{n,m}$ for some $u\in N_{\a_r}(F)$}}.
\end{align*}
Then the observations (i), (ii) mentioned at the beginning of this section adapt to $M^s_{r,m}$.

Let $\Gamma_{r,m}\subset\GL_r(\frak{o})$ be the open compact subgroups defined by 
\[
\Gamma_{r,m}
=
\bordermatrix{
              & (r-1)     & 1       \cr
    (r-1) & \mathfrak{o} &\mathfrak{o}   \cr
    1     &\mathfrak{p}^m& \mathfrak{u}^m \cr
            }\cap {\rm GL}_r(\frak{o}).
\]
These subgroups were introduced in \cite{Casselman1973} and \cite{JPSS1981} 
(see also \cite{Jacquet2012}) in connection with the newform theory for $\GL_r(F)$. In what follows, we also denote 
\[
\b{s}_{r,i}=x_{-\ep_r}\left(\varpi^i\right)
\]
for $r,i\in\bbZ$ with $1\le r\le n$.

\begin{cor}\label{C:gp of int}
Let $e=0,1$ such that $m\equiv e\pmod 2$. Then we have
\[
M^{\b{s}_{r,i}}_{r,m}
=
\stt{\diag{a,h,a^*}\in M_{\a_r}(F)\mid a\in\Gamma_{r,\lceil\frac{m}{2}\rceil-i},\, h\in K^0_{n-r, 2i-e}}
\simeq
\Gamma_{r,\lceil\frac{m}{2}\rceil-i}\,\x\, K^0_{n-r, 2i-e}
\]
for $e\le i\le\lceil\frac{m}{2}\rceil$.
\end{cor}

\begin{proof}
Since
\[
N_{\a_r}(F)=w_{r,e}\,\b{N}_{\a_r}(F)\,w^{-1}_{r,e}
\]
and $w_{r,e}$ normalizes $K^0_{n,m}$, we find that
\[
g\in \b{M}^{s_{r,i}}_{r,m}
\quad\Longleftrightarrow\quad
w_{r,e}gw_{r,e}^{-1}
\in M^{s'_{r,i}}_{r,m}=M^{\b{s}_{r,e+i}}_{r,m},
\]
where $0\le i\le\lfloor\frac{m}{2}\rfloor$ and 
\[
s'_{r,i}=w_{r,e}s_{r,i}w_{r,e}^{-1}=w_{r,e}x_{\ep_r}(\varpi^i)w_{r,e}^{-1}=x_{-\ep_r}\left((-1)^r\varpi^{e+i}\right).
\]
Since
\[
\Gamma_{r,m}=\stt{{}^tk^{-1}\mid k\in\Gamma^{'}_{r,m}},
\]
the corollary now follows from \propref{P:gp of int K^0}.
\end{proof}

\section{Proof of \thmref{T:main}}\label{S:proof}
We prove \thmref{T:main} in this section. The first step is to reduce the argument to the case of tempered representations. To this end, 
we make use of the results from the previous sections together with the newform theory for $\GL_r(F)$.

Let $\tau$ be an irreducible generic representation of $\GL_r(F)$. Denote by $c_\tau$ the smallest 
integer $m\ge 0$ such that $\tau^{\Gamma_{r,m}}\ne 0$. By results of Casselman (\cite{Casselman1973}) and
Jacquet--Piatetski-Shapiro--Shalika (\cite{JPSS1981}, see also \cite{Jacquet2012}, \cite{Matringe2013}), such an integer always 
exists, and moreover
\[
\dim_{\bbC}\tau^{\Gamma_{r,c_\tau}}=1.
\]
In addition, $c_\tau$ is reflected in the associated $\epsilon$-factor:
 \[
 \epsilon(s,\tau,\psi)=cq^{c_\tau\left(s-\frac{1}{2}\right)},
 \] 
 where $c\in\bbC^\x$ is a constant independent of $s$.

\subsection{Reduction to the tempered case}\label{SS:reduction}
Let $1\le r\le n$, and $\tau$ be an irreducible generic representation of $\GL_r(F)$. Let $\pi_0$ be an irreducible representation of 
$\SO_{2(n-r)+1}(F)$. Denote by 
\[
\pi={\rm Ind}_{P_{\a_r}(F)}^{\SO_{2n+1}(F)}\left(\tau\boxtimes\pi_0\right)
\]
the normalized induced representation of $\SO_{2n+1}(F)$. The following lemma is key to the reduction.

\begin{lm}\label{L:reduction}
Suppose that there exist an integer $c_0\ge 0$ and $\e_0=\pm 1$ such that  
\begin{itemize}
\item the subspaces satisfy $\pi_0^{K_{n-r,m}}=0$ for $0\le m<c_0$ and $\dim_\bbC\pi_0^{K_{n-r, c_0}}\le 1$,
\item when $\pi_0^{K_{n-r,c_0}}\ne 0$, the action of $J_{n-r,c_0}/K_{n-r,c_0}$ on $\pi_0^{K_{n-r,c_0}}$ is given by the scalar $\e_0$.
\end{itemize}
Then we have
\begin{itemize}
\item[(1)] the subspaces satisfy $\pi^{K_{n,m}}=0$ for $0\le m<c$ and $\dim_\bbC\pi^{K_{n, c}}=\dim_\bbC\pi_0^{K_{n-r,c_0}}$,
\item[(2)] when $\pi^{K_{n,c}}\ne 0$, the action of $J_{n,c}/K_{n,c}$ on $\pi^{K_{n,c}}$ is given by the scalar 
$\e_0\omega_{\tau}(-1)$,
\end{itemize}
where $c:=c_0+2c_\tau$ and $\omega_\tau$ is the central character of $\tau$.
\end{lm}

\begin{proof}
Since $K_{n,m}$ is conjugate to $K_{n,m}^0$, the proof reduces to verifying the analogous assertion for $K^0_{n,m}$.
Suppose that $1\le r\le n-1$. Then by \corref{C:double coset} and \corref{C:gp of int}, we have 
\begin{align*}
\pi^{K^0_{n,m}}
\simeq
\bigoplus_{s\in \Sigma_{r,m}}(\tau\boxtimes\pi_0)^{M^s_{r,m}}
\simeq
\bigoplus_{i=e}^{\lceil\frac{m}{2}\rceil} \tau^{\Gamma_{r,\lceil\frac{m}{2}\rceil-i}}\ot\pi_0^{K^0_{n-r, 2i-e}},
\end{align*}
where $e=0,1$ is such that $m\equiv e\pmod 2$. By the result of \cite{JPSS1981} and the assumption on $\pi_0$, we find that 
\[
\tau^{\Gamma_{r,\lceil\frac{m}{2}\rceil-i}}\ot\pi_0^{K^0_{n-r, 2i-e}}\ne 0
\]
implies $\lceil\frac{m}{2}\rceil-i\ge c_\tau$ and $2i-e\ge c_0$. In turn, these inequalities yield $m\ge c_0+2c_\tau=c$. This shows that 
$\pi^{K^0_{n,m}}=0$ for $0\le m<c$.

On the other hand, when $m=c$, the result in $loc$. $cit$. together with the assumption on $\pi_0$ imply 
\begin{align*}
\dim_{\bbC}\pi^{K_{n,c}^0}
&=
\sum_{i=e}^{\lceil\frac{c}{2}\rceil}\dim_{\bbC}\left(\tau^{\Gamma_{r,\lceil\frac{c}{2}\rceil-i}}\ot\pi_0^{K^0_{n-r, 2i-e}}\right)\\
&\overset{i_0=\lceil\frac{c_0}{2}\rceil}{=}
\dim_{\bbC}\left(\tau^{\Gamma_{r,\lceil\frac{c}{2}\rceil-i_0}}\ot\pi_0^{K^0_{n-r, 2i_0-e}}\right)\\
&\quad=
\dim_{\bbC}\left(\tau^{\Gamma_{r, c_\tau}}\ot\pi_0^{K^0_{n-r, c_0}}\right)\\
&\quad=\dim_\bbC\pi_0^{K_{n-r,c_0}}.
\end{align*}
This establishes the assertion $(1)$ for the case $1\le  r\le n-1$.


We proceed to prove the assertion $(2)$, again for the case $1\le r\le n-1$. Note that the above argument also implies
\[
\pi^{K_{n,c}}=\bbC f_\pi,
\]
where $f_\pi$ is characterized by 
\[
\supp(f_\pi)=P_{\a_r}(F)sK^0_{n,c}
\quad\text{and}\quad
 f_\pi\left(s\right)=v_\tau\ot v_{\pi_0}, 
\]
with $s:=x_{-\ep_r}\left(\varpi^{\lceil\frac{c_0}{2}\rceil}\right)$  and  $v_\tau\in\tau^{\Gamma_{r,c_\tau}}$,  
$v_{\pi_0}\in\pi_0^{K^0_{n-r, c_0}}$ fixed basis vectors. 

Since the case $c=0$ follows easily from the fact that $J^0_{n,0}=K^0_{n,0}=K_{n,0}$, we may assume that $c>1$. 
To compute the action of $J^0_{n,c}/K^0_{n,c}$ on $\pi^{K^0_{n,c}}$, it suffices to compute the action of any element 
$w\in J^0_{n,c}\setminus K^0_{n,c}$ on the same space. To this end, take $w=w_{\ep_n,e}$ (see \eqref{E:property of K^0}). 
Suppose that the action of $J^0_{n,c}/K^0_{n,c}$ on $\pi^{K^0_{n,c}}$ is given by the scalar $\e=\pm 1$, i.e.
\[
\pi(w)f_\pi=\e f_\pi.
\]
To determine $\e$, it suffices to compute the value $\pi(w)f_\pi(s)=f_\pi\left(sw\right)$.

Since $r<n$, a simple computation shows
\[
sw
=
x_{-\ep_r}\left(\varpi^{\lceil\frac{c_0}{2}\rceil}\right)w_{\ep_n,e}
=
w_{\ep_n,e}x_{-\ep_r}\left(-\varpi^{\lceil\frac{c_0}{2}\rceil}\right)
=
\diag{-I_r, w', -I_r}s,
\]
where $w'\in\SO_{2(n-r)+1}(F)$ is the element analogue to $w=w_{\ep_n,e}\in\SO_{2n+1}(F)$. Since
$c_0\equiv c\equiv e\pmod 2$, we have $w'\in J^0_{n-r,c_0}\setminus K^0_{n-r, c_0}$, and the above identity implies
\begin{align*}
\e\left(v_\tau\ot v_{\pi_0}\right)
=f_{\pi}(sw)
&=f_\pi\left(\diag{-I_r, w', -I_r}s\right)\\
&=\tau(-I_r)v_\tau\ot\pi_0(w')v_{\pi_0}
=\e_0\omega_\tau(-1)\left(v_\tau\ot v_{\pi_0}\right).
\end{align*}
This proves the assertion $(2)$ for the case $1\le r<n$.

Suppose that $r=n$. Note that in this case, $\pi_0$ is the trivial representation of the trivial group $\SO_1(F)$. It follows that 
$c_0=0$, $\e_0=1$ and $\dim_\bbC\pi_0^{K^0_{n-r,c_0}}=1$. Thus our task is to show that $\pi^{K^0_{n,m}}=0$ for 
$0\le m<c=2c_\tau$, that $\dim_\bbC\pi^{K^0_{n,c}}=1$, and that the action of $J^0_{n,c}/K^0_{n,c}$ on $\pi^{K^0_{n,c}}$ is given 
by the scalar $\omega_{\tau}(-1)$.

By \corref{C:double coset} and \corref{C:gp of int}, we have 
\begin{align*}
\pi^{K^0_{n,m}}
\simeq
\bigoplus_{s\in \Sigma_{n,m}}\tau^{M^s_{n,m}}
\simeq
\bigoplus_{i=e}^{\lceil\frac{m}{2}\rceil}\tau^{\Gamma_{n,\lceil\frac{m}{2}\rceil-i}}
\oplus
\bigoplus_{j=1}^{\lceil\frac{m}{2}\rceil}\tau^{\Gamma_{n,\lceil\frac{m}{2}\rceil-j}}
\end{align*}
where $e=0,1$ satisfies $m\equiv e\pmod 2$. By the result of \cite{JPSS1981}, we find that 
\[
\tau^{\Gamma_{n,\lceil\frac{m}{2}\rceil-i}}\ne 0
\]
implies $\lceil\frac{m}{2}\rceil-i\ge c_\tau$. Therefore, we must have $m\ge 2c_\tau=c$, and hence $\pi^{K^0_{n,m}}=0$ for 
$0\le m<c$.

On the other hand, when $m=c=2c_\tau$, the result in $loc$. $cit$. also implies
\begin{align*}
\dim_\bbC\pi^{K^0_{n,c}}
=
\dim_\bbC\tau^{\Gamma_{n,c_\tau}}
+
2\sum_{i=1}^{c_\tau}\dim_\bbC\tau^{\Gamma_{n,c_\tau-i}}=1.
\end{align*}
This prove the assertion $(1)$. Note that the above argument also shows that 
\[
\pi^{K^0_{n,c}}=\bbC f_\pi,
\]
where $f_\pi$ is characterized by 
\[
\supp(f_\pi)=P_{\a_n}(F)sK^0_{n,c}
\quad\text{and}\quad
f_\pi(s)=v_\tau,
\]
with $s:=x_{-\ep_n}(1)$ and $v_\tau\in\tau^{\Gamma_{n,c_\tau}}$ a fixed basis vector.

To prove the assertion $(2)$ in this case, we may again assume $c>0$, as the case $c=0$ follows easily from the fact that 
$J^0_{n,0}=K^0_{n,0}=K_{n,0}$. For $c>0$, let $w=w_{\ep_n}\in J^0_{n,c}\setminus K^0_{n,c}$. Then we have
\[
\pi(w)f_\pi=\e f_\pi
\]
for some $\e=\pm 1$. It then follows from the identity
\[
sw
=
x_{\ep_n}(1)w_{\ep_n}
=
x_{\ep_n}(1)x_{-\ep_n}(-1)
=
x_{\ep_n}(1)\diag{-I_n,1,-I_n} s
\]
that
\[
\e v_\tau
=
f_\pi(sw)
=
f_\pi\left(x_{\ep_n}(1)\diag{-I_n,1,-I_n} s\right)
=
\tau(-I_n)v_\tau
=\omega_\tau(-1)v_\tau.
\]
This proves the assertion $(2)$ for the case $r=n$, and hence complete the proof of the lemma.
\end{proof}

As a corollary, we obtain:

\begin{cor}
If \thmref{T:main} $(1)$ and $(2)$ hold for all irreducible generic tempered representations, then they hold for all irreducible
generic representations. 
\end{cor}

\begin{proof}
Let $\pi$ be an irreducible generic representation of $\SO_{2n+1}(F)$ that is non-tempered.
Then by Langlands' classification (\cite{Silberger1978}) together with the standard module conjecture proved in 
\cite{CasselmanShahidi1998} and  \cite{Muic2001}, we have
\[
\pi
=
{\rm Ind}_{P_{S}(F)}^{\SO_{2n+1}(F)}(\tau_1\boxtimes\cdots\boxtimes\tau_k\boxtimes\pi_0)
\quad
(\text{normalized induction}),
\]
where $S=\stt{\a_{r_1}, \a_{r_1+r_2}, \ldots, \a_{r_1+r_2+\cdots +r_k}}$ is a subset of $\Delta_n$ with $1\le r:=r_1+r_2\cdots r_k\le n$.
The Levi subgroup of $P_{S}$ isomorphic to 
\[
{\rm GL}_{r_1}\x\cdots\x{\rm GL}_{r_k}\x\SO_{2(n-r)+1}.
\] 
Here each $\tau_j$ ($1\le j\le n$) is an irreducible essentially square integrable representations of ${\rm GL}_{r_j}(F)$, 
and $\pi_0$ is an irreducible tempered generic representation of $\SO_{2(n-r)+1}(F)$.

Now, the corollary follows from induction on $k$, invoking \lmref{L:reduction}, together with the relations
\[
c_\pi=c_{\pi_0}+2c_{\tau_1}+\cdots+2c_{\tau_k}
\quad\text{and}\quad
\e_\pi=\e_{\pi_0}\omega_{\tau_1}(-1)\cdots\omega_{\tau_k}(-1),
\]
where $\omega_{\tau_j}$ ($1\le j\le k$) denotes the central character of $\tau_j$, and by applying induction in stages.
\end{proof}

\subsection{Local Rankin-Selberg integrals}\label{SS:RS int}
In the remainder of this section, we focus on proving \thmref{T:main} for tempered representations. To this end, we require the 
local Rankin-Selberg integrals for $\SO_{2n+1}\x\GL_r$ developed by Ginzburg (\cite{Ginzburg1990}) and Soudry  
(\cite{Soudry1993}, \cite{Soudry2000}). 

Let $Z_r\subset\GL_r$ be the upper triangular maximal unipotent subgroup. 
Define a non-degenerate character of $Z_r(F)$ by 
\[
\psi^{-1}_{Z_r}(z)
=
\psi^{-1}(z_{12}+z_{23}+\cdots+z_{r-1, r}),
\]
for $z=(z_{ij})\in Z_r(F)$. Let $\tau$ be an irreducible generic representation of $\GL_r(F)$. We fix a nonzero element 
\[
\Lambda_{\tau,\psi^{-1}}\in{\rm Hom}_{Z_r(F)}(\tau,\psi^{-1}_{Z_r}). 
\]
Let $s$ be a complex number. Denote by $\tau_s$ the representation of $\GL_r(F)$ on the same space of $\tau$ with the action 
$\tau_s(a)=\tau(a)|\det(a)|_F^{s-\frac{1}{2}}$.

Suppose that $1\le r\le n$. Recall that we have identified the split group $\SO_{2r}(F)$ as a 
subgroup of $\SO_{2n+1}(F)$ via the embedding \eqref{E:embedding}.
Let $Q_r\subset \SO_{2r}$ be a Siegel parabolic subgroup with the Levi decomposition $L_r\ltimes Y_r$, where
\[
L_r(F)
=
\stt{m_r(a)=\pMX{a}{}{}{a^*}\mid \text{$a\in\GL_r(F)$}}\cong\GL_r(F)
\]
and 
\[
Y_r(F)
=
\stt{\pMX{I_r}{b}{}{I_r}\mid\text{$b\in\M_{r\x r}(F)$ with $b=-J_r{}^t b J_r$}}.
\]

Denote by
\[
\rho_{\tau, s}={\rm Ind}_{Q_r(F)}^{\SO_{2r}(F)}(\tau_s)
\]
a normalized induced representation of $\SO_{2r}(F)$. 
Its underlying space, denoted $I_r(\tau,s)$ of $\rho_{\tau,s}$, consists of smooth functions $\xi_s: \SO_{2r}(F)\to\tau$ satisfying 
\[
\xi_s(m_r(a)uh)=\delta_{Q_r}^{\frac{1}{2}}(m_r(a))\tau_s(a)\xi_s(h),
\]
for $a\in\GL_r(F)$, $u\in Y_r(F)$ and $h\in\SO_{2r}(F)$. Note that the modulus character $\delta_{Q_r}$ of $Q_r$ is given by 
$\delta_{Q_r}(m_r(a))=|\det(a)|_F^{r-1}$. 

Let $\pi$ be an irreducible generic representation of $\SO_{2n+1}(F)$, and fixed a nonzero Whittaker functional 
\[
\Lambda_{\pi,\psi}\in{\rm Hom}_{U_n(F)}(\pi,\psi_{U_n}).
\] 
Then the local Rankin-Selberg integral $\Psi_{r}(v\ot\xi_s)$ attached to $v\in\pi$ and $\xi_s\in I_r(\tau,s)$ is defined by
\[
\Psi_{r}(v\ot\xi_s)
=
\int_{\left(\SO_{2r}(F)\cap U_n(F)\right)\backslash \SO_{2r}(F)}\int_{\M_{(n-r)\x r}(F)}
W_v(\hat{x}h)f_{\xi_s}(h)dxdh,
\]
where $W_v(g)=\Lambda_{\pi,\psi}(\pi(g)v)$ is the Whittaker function associated to $v$, 
$f_{\xi_s}(h)=\Lambda_{\tau,\psi^{-1}}(\xi_s(h))$, and 
\[
\hat{x}=
\begin{pmatrix}I_r&&&&\\x&I_{n-r}&&\\&&1\\&&&I_{n-r}&\\&&&-J_r{}^txJ_{n-r}&I_r\end{pmatrix}
\in\SO_{2n+1}(F)
\]
for $x\in{\rm Mat}_{(n-r)\x r}(F)$.

As usual, these integrals converge absolutely for $\Re(s)\gg 0$, admit meromorphic continuation to the entire complex plane, and 
yield rational functions in $q^{-s}$. Furthermore, they satisfy a functional equation relating $s$ and $1-s$, and there exist $v$ and 
$\xi_s$ such that $\Psi_r(v\ot\xi_s)\equiv 1$.

\begin{lm}\label{L:converge}
If both $\pi$ and $\tau$ are tempered, then the integrals $\Psi_r(v\ot\xi_s)$ converge absolutely for $\Re(s)>0$.
\end{lm}

\begin{proof}
From the proof of \cite[Proposition 4.2]{Soudry1993}, it suffices to estimate the  integral:
\[
\prod_{j=1}^r\int_{F^\x}|\varphi_j(y_j)\chi_j(y_j)||y_j|_F^{js}d^{\x}y_j,
\]
where each $\varphi_j$ is a Bruhat-Schwartz function on $F$,  and each $\chi_j$ is a character of $F^\x$ depending 
on $\pi$ and $\tau$. In fact, these functions and characters arise from the asymptotic behavior of the Whittaker functions associated 
with $\pi$ and $\tau$ (see \cite[Proposition 2,2]{Soudry1993}, \cite[Proposition 2.2]{JPSS1979}). 

Since both $\pi$ and $\tau$ are tempered, we have
\[
|\chi_j(y_j)|=|y_j|_F^{e_j}
\]
for some $e_j\ge 0$,  by a result of Waldspurger (\cite[Proposition III.2.2]{Waldspurger2003}). The lemma then follows immediately. 
\end{proof}

To state the next lemma, recall that $H_{r,m}$ denotes the conjugates of hyperspecial maximal compact subgroups 
of $\SO_{2r}(F)$ defined by \eqref{E:H}. Suppose that $\tau$ is unramified, and let $v_\tau\in\tau$ be a spherical vector. 
For each $m\ge 0$, let $\xi_{\tau,s}^m\in I_r(\tau,s)$ denote the unique section satisfying: 
\begin{itemize}
\item $\xi^m_{\tau,s}$ is right $\left(K_{n,m}\cap\SO_{2r}(F)\right)$-invariant, and
\item $\xi_{\tau,s}^m(I_{2r})=v_\tau$.
\end{itemize}

Recall that $T_n$ denotes the diagonal torus of $\SO_{2n+1}$. Define $T_r(F)=T_n(F)\cap\SO_{2r}(F)$, which is the diagonal torus 
of $\SO_{2r}(F)$. Now, the lemma can be stated as follows.

\begin{lm}\label{L:van}
Let $v\in\pi^{K_{n,m}}$ and suppose that $W_v$ vanishes identically on $T_r(F)$. Then we have 
$\Psi_r\left(v\ot\xi_{\tau,s}^m\right)=0$.
\end{lm}

\begin{proof}
Fix $s_0\in\bbC$ such that $\Psi_r(v\ot\xi^m_{\tau,s_0})$ converges absolutely. 
By the Iwasawa decomposition of $\SO_{2r}(F)$, we have 
\[
\Psi_r\left(v\ot\xi_{\tau,s_0}^m\right)
=
\int_{T_r(F)}\int_{\M_{(n-r)\x r}(F)} W_v(\hat{x}t)f_{\xi_{\tau,s_0}^m}(t)\delta^{-1}(t)dxdt,
\]
where $\delta$ denotes the modulus character of the upper triangular Borel subgroup of $\SO_{2r}$. In particular, this already implies
that $\Psi_n(v\ot\xi^m_{\tau,s_0})=0$. By uniqueness of meromorphic continuation, we conclude that $\Psi_n(v\ot\xi^m_{\tau,s})=0$.

Suppose that $r<n$. Then since $T_r(F)$ normalizes the subgroup $\stt{\hat{x}\mid x\in\M_{(n-r)\x r}(F)}$ of $\SO_{2n+1}(F)$, we 
may change variables to obtain
\[
\Psi_r\left(v\ot\xi_{\tau,s_0}^m\right)
=
\int_{T_r(F)}\int_{\M_{(n-r)\x r}(F)} W_v(t\hat{x})f_{\xi_{\tau,s_0}^m}(t)\nu(t)\delta^{-1}(t)dxdt,
\]
where $\nu$ is the character on $T_r(F)$ arising from this change the variables. Now, we apply \cite[Lemma 6.5]{YCheng2022}, 
which asserts that 
\[
W_v(t\hat{x})=0
\]
for all $t\in T_r(F)$, whenever $x\in\M_{(n-r)\x r}(F)\setminus\M_{(n-r)\x r}(\frak{o})$. Together with the fact that 
\[
\stt{\hat{x}\mid x\in\M_{(n-r)\x r}(\frak{o})}
\subset
K_{n,m},
\]
we conclude 
\begin{align*}
\Psi_r\left(v\ot\xi_{\tau,s_0}^m\right)
&=
\int_{T_r(F)}\int_{\M_{(n-r)\x r}(F)} W_v(t\hat{x})f_{\xi_{\tau,s_0}^m}(t)\nu(t)\delta^{-1}(t)dxdt\\
&=
\int_{T_r(F)}\int_{\M_{(n-r)\x r}(\frak{o})} W_v(t\hat{x})f_{\xi_{\tau,s_0}^m}(t)\nu(t)\delta^{-1}(t)dxdt\\
&=
c\int_{T_r(F)}W_v(t)f_{\xi_{\tau,s_0}^m}(t)\nu(t)\delta^{-1}(t)dt=0,
\end{align*}
where $c$ is the volume of $\M_{(n-r)\x r}(\frak{o})$. This proves the lemma.
\end{proof}

Now we are ready to prove the following key lemma.

\begin{lm}\label{L:inj}
Let $v\in\pi^{K_{n,m}}$. Suppose that $\pi$ is tempered and $W_v$ vanishes identically on $T_n(F)$. Then we have $v=0$.
\end{lm}

\begin{proof}
Suppose, for the sake of contradiction, that $W_v$ vanishes identically on $T_n(F)$ but $v\neq 0$.
Consider the matrix coefficient
\[
f_v(g)=\langle\pi(g)v,v\rangle_\pi
\]
associated to $v$, where $\langle\cdot,\cdot\rangle_\pi$ is the $\SO_{2n+1}(F)$-equivariant Hermitian pairing on $\pi$.

Since $f_v(I_{2n+1})\ne 0$, a lemma of Gan-Savin (\cite[Lemma 12.5]{GanSavin2012}) implies that there exist
\begin{itemize}
\item a tempered representation $\pi'$ of the split group $\SO_{2n}(F)$, and
\item a matrix coefficient $f_{v'}$ associated to $v'\in\pi'$, 
\end{itemize}
such that 
\begin{equation}\label{E:non-van}
\int_{\SO_{2n}(F)} f_v(h)\overline{f_{v'}(h)}dh\ne 0.
\end{equation}

We may assume that $v'$ is fixed by $H_{n,m}$, which implies that $\pi'$ is $H_{n,e}$-unramified, where $e=0,1$ satisfies 
$m\equiv e\pmod 2$. Consequently, we may further assume 
\[
\tilde{\pi}'\subset I\left(\tau,\frac{1}{2}\right)
\]
is an irreducible summand, and
\[
v'=\xi^m_{\tau,\frac{1}{2}},
\]
where $\tilde{\pi}'$ denotes the contragredient of $\pi'$ and $\tau$ is an irreducible unramified tempered representation of 
$\GL_n(F)$. Observe that $\tau$ is necessarily generic.

By \lmref{L:converge}, the map 
\begin{equation}\label{E:RS map}
w\ot\xi_{\frac{1}{2}}\mapsto\Psi_n\left(w\ot\xi_{\frac{1}{2}}\right)
\end{equation}
is well-defined. If $\xi_{\frac{1}{2}}\in\tilde{\pi}'$, then it clear belongs to ${\rm Hom}_{\SO_{2n}(F)}(\pi\boxtimes\tilde{\pi}',\bbC)$. 
Assume for a moment that this map, when restricted to $\tilde{\pi}'$, is non-zero. Then since the map
\[
w\ot w'\mapsto\int_{\SO_{2n}(F)}\langle\pi(h)w,v\rangle_\pi\overline{\langle\pi'(h)w',v'\rangle_\pi'}dh
\]
also yields a non-zero element in the same space by \eqref{E:non-van}, and  the space 
${\rm Hom}_{\SO_{2n}(F)}(\pi\boxtimes\tilde{\pi}',\bbC)$ is at most one-dimensional (\cite{AGRS2010}), it follows that  
\[
\int_{\SO_{2n}(F)} f_v(h)\overline{f_{v'}(h)}dh
=
c\,\Psi_n\left(v\ot\xi^m_{\tau,\frac{1}{2}}\right),
\]
for some non-zero constant $c$. By \lmref{L:van} and \eqref{E:non-van}, we obtain the desired contradiction. 

It remains to show that the map \eqref{E:RS map} is non-zero on $\pi\ot\tilde{\pi}'$. By \cite[Proposition 6.1]{Soudry1993}, there exist 
$w\in\pi$ and a section $\xi_s\in I(\tau,s)$ such that 
\[
\Psi_n(w\ot\xi_s)\equiv 1.
\]
In particular, the map \eqref{E:RS map} is non-zero on $I\left(\tau,\frac{1}{2}\right)$, and if $I\left(\tau,\frac{1}{2}\right)$ is irreducible, 
then we are done. Suppose instead that $I\left(\tau,\frac{1}{2}\right)$ is reducible. Then by 
\cite[Theorem 6.8]{Goldberg1994}, 
\[
I\left(\tau,\frac{1}{2}\right)\simeq\tilde{\pi}'\oplus\pi''
\]
where $\pi''$ is an irreducible $H_{n,e'}$-unramified tempered representation of $\SO_{2n}(F)$ with 
$e'=0,1$ and $e'\ne e$. Since $\tilde{\pi}'$ and $\pi''$ lie in the same (generic) $L$-packet, the Gross-Prasad conjecture 
(\cite[Conjecture 8.6]{GrossPrasad1992}, now a Theorem of Waldspurger) implies 
\[
\dim_\bbC{\rm Hom}_{\SO_{2n}(F)}(\pi\ot\tilde{\pi}',\bbC)
+
\dim_\bbC{\rm Hom}_{\SO_{2n}(F)}(\pi\ot\pi'',\bbC)\le 1.
\]
Since ${\rm Hom}_{\SO_{2n}(F)}(\pi\ot\tilde{\pi}',\bbC)\ne 0$, we conclude that 
\[
{\rm Hom}_{\SO_{2n}(F)}(\pi\ot\pi'',\bbC)=0.
\]
Consequently, the map \eqref{E:RS map} vanishes on $\pi\ot\pi''$, and the proof is complete.
\end{proof}

We conclude this subsection with the following proposition, which plays a key role in the proof of the tempered case. Let 
$\cS_r$ be the $\bbC$-algebra of symmetric polynomials in the $r$ variables $X_1,\ldots, X_r$. Recall that the $\epsilon$-factor
$\epsilon(s,\phi_\pi,\psi)$ associated to the $L$-parameter $\phi_\pi$ of $\pi$ is of the form
\[
\epsilon(s,\phi_\pi,\psi)
=
\e_\pi q^{-c_\pi\left(s-\frac{1}{2}\right)}
\]
for some $\e_\pi=\pm 1$ and $c_\pi\ge 0$.

\begin{prop}\label{P:key}
For each $1\le r\le n$ and $m\ge 0$, there exists a linear map 
\[
\Xi_{r,m}:\pi^{K_{n,m}}\longto\cS_r,
\quad
v\mapsto\Xi_{r,m}(v;X_1,\ldots, X_r),
\]
satisfying the following properties.

\begin{itemize}
\item[(1)] 
If $\tau$ is an irreducible unramified generic representation of $\GL_r(F)$ with the Satake parameters $z_1,\ldots, z_r$, then
\[
\Xi_{r,m}\left(v;q^{-s+\frac{1}{2}}z_1,\ldots, q^{-s+\frac{1}{2}}z_r\right)
=
\frac{L\left(2s,\phi_\tau,\wedge^2\right)\Psi_r\left(v\ot\xi^m_{\tau,s}\right)}{L(s,\phi_\pi\ot\phi_\tau)},
\]
where $\phi_\tau$ is the $L$-parameter of $\tau$.
\item[(2)]
The functional equation
\[
\Xi_{r,m}\left(w_{r,m}v; X_1^{-1},\ldots, X_r^{-1}\right)
=
\e_\pi^r\left(X_1\cdots X_r\right)^{c_\pi-m}
\Xi_{r,m}\left(v;X_1,\ldots, X_r\right)
\]
holds.
\item[(3)]
The relation
\[
\Xi_{r,m}(v;X_1,\ldots,X_{r-1},0)=\Xi_{r-1,m}(v;X_1,\ldots, X_{r-1})
\]
holds for $r\ge 2$.

\item[(4)]
The kernel of $\Xi_{r,m}$ is given by
\[
\ker\left(\Xi_{r,m}\right)
=
\stt{v\in\pi^{K_{n,m}}\mid\text{$W_v$ vanishes identically on $T_r(F)$}}.
\]
\end{itemize}
\end{prop}

\begin{proof}
This is \cite[Proposition 6.7]{YCheng2022}. We note that $\cS_r$ in $loc.$ $cit.$ denotes the $\bbC$-algebra of symmetric 
polynomials in $X^{\pm}_1,\ldots, X_r^{\pm}$. The reason is that, when $r=n$, the spaces considered in $loc.$ $cit.$ are
\[
\pi^{H_{n,m}}
\]
rather than $\pi^{K_{n,m}}$ in our setting. However, one can show that (see \cite[Equation (7.2)]{YCheng2022}) the image of the maps 
constructed in $loc.$ $cit.$, upon restriction to $\pi^{K_{n,m}}$, lies inside the $\bbC$-algebra of symmetric polynomials in 
$X_1,\ldots, X_r$.

Another point worth noting is that the element $u_{n,r,m}$ appearing in the functional equation in $loc.$ $cit.$ is different from 
$w_{r,m}$ in our setting. Nevertheless, one can check easily that $w_{r,m}^{-1}u_{n,r,m}\in K_{n,m}$, so the functional equation can
be written in the form stated here.
\end{proof}

\subsection{Proof of the tempered case}\label{SS:proof}
Let $\pi$ be an irreducible generic tempered representation of $\SO_{2n+1}(F)$. Our goal is to prove the following:
\begin{itemize}
\item[(1)] the subspaces satisfy $\pi^{K_{n,m}}=0$ for $0\le m<c_\pi$ and $\dim_\bbC\pi^{K_{n,c_\pi}}\le 1$;
\item[(2)] if $\pi^{K_{n,c_\pi}}\ne 0$, then the action of $J_{n,c_\pi}/K_{n,c_\pi}$ on $\pi^{K_{n,c_\pi}}$ is given by the scalar $\e_\pi$;
\item[(3)] if $\pi^{K_{n,c_\pi}}\ne 0$, then the natural pairing of one-dimensional spaces 
\[
{\rm Hom}_{K_{n,c_\pi}}\left(1,\pi\right)\,\x\,{\rm Hom}_{U_n}\left(\pi,\psi_{U_n}\right)\longto\bbC,
\]
is non-degenerate.
\end{itemize}

To this end, we apply the linear maps $\Xi_{r,m}$ from \propref{P:key}. We begin with the proof of $(1)$. By the functional equation
\[
\Xi_{n,m}\left(v;X_1,\ldots, X_r\right)
=
\e_\pi^n\left(X_1\cdots X_r\right)^{m-c_\pi}
\Xi_{n,m}\left(w_{r,m}v; X_1^{-1},\ldots, X_r^{-1}\right),
\]
and the fact that the image of $\Xi_{n,m}$ lies in $\cS_r$, we immediately deduce
\begin{itemize}
\item $\Xi_{n,m}(v;X_1,\ldots, X_n)=0$ for $0\le m<c_\pi$, 
\item $\Xi_{n, c_\pi}(v;X_1,\ldots, X_n)\in\bbC$.
\end{itemize}
Since the maps $\Xi_{n,m}$ are injective by \propref{P:key} (4) and \lmref{L:inj}, it follows that 
$\pi^{K_{n,m}}=0$ for $0\le m<c_\pi$ and $\dim_\bbC\pi^{K_{n,c_\pi}}\le 1$. This verifies $(1)$.

To prove $(2)$, observe that when $c_\pi=0$, we have $J_{n,0}=K_{n,0}$ and $\e_\pi=1$, so the assertion is immediate. 
In fact, we have already noted that the conjecture holds for unramified $\pi$. So suppose that $c_\pi>0$, and the action of 
$J_{n,c_\pi}/K_{n,c_\pi}$ on $\pi^{K_{n,c_\pi}}$ is given by the scalar $\e=\pm 1$. Since 
$w_{1,c_\pi}\in J_{n,c_\pi}\setminus K_{n,c_\pi}$, this is equivalent to 
\[
\pi(w_{1,c_\pi})v=\e v,
\]
for any non-zero vector $v\in\pi^{K_{n,c_\pi}}$. 

From the proof of $(1)$, we know that $\Xi_{n,c_\pi}(v;X_1,\ldots, X_n)=c_0$ is a non-zero constant. It then follows from
\propref{P:key} $(3)$ that
\begin{equation}\label{E:c}
\Xi_{1,c_\pi}(v;X_1)=\Xi_{n,c_\pi}(v;X_1, \underbrace{0,\ldots,0}_{n-1})=c_0\ne 0.
\end{equation}
Applying the functional equation 
\[
\Xi_{1,c_\pi}(w_{1,c_\pi}v;X_1^{-1})=\e_\pi\Xi_{1,c_\pi}(v;X_1),
\]
we conclude that $\e=\e_\pi$. This shows $(2)$.

By \propref{P:key} $(1)$ and \eqref{E:c}, we have
\[
c_0=\Xi_{1,c_\pi}\left(v;q^{-s+\frac{1}{2}}\right)
=\frac{Z(s,v)}{L(s,\phi_\pi)}
=\lambda_{\pi,\psi}(v)+c_1q^{-s}+c_2q^{-2s}+\cdots
\]
after expanding the right-hand side as a power series in $q^{-s}$. Here 
\[
Z(s,v)
=
\int_{F^\x}\int_{\M_{(n-1)\x 1}(F)}
W_v(\hat{x}\epsilon_1^*(y))|y|_F^{s-\frac{1}{2}}dxd^{\x} y
\]
is the Rankin-Selberg integral attached to $v$ for $r=1$ and $\tau$ is the trivial character of $F^\x$ 
(see \cite[Remark in \S 4.2]{YCheng2022}). We note that such an 
expansion can be derived from \cite[Lemma 6.6]{YCheng2022}. Now the assertion $(3)$ follows immediately. This completes the 
proof of the tempered case, and hence the proof of \thmref{T:main}.\qed

\end{document}